\documentclass[bibtex,12pt,en]{elegantpaper}
\usepackage{extarrows}
\usepackage{esint}
\usepackage{mathrsfs}
\usepackage{mathtools}
\usepackage{microtype}
\numberwithin{equation}{section}
\allowdisplaybreaks[4]
\everymath{\displaystyle}

\newcommand{\dd}{\,dx}
\newcommand{\ds}{\,dS}
\newcommand{\Div}{\operatorname{div}}

\newcommand{\dist}{\operatorname{dist}}
\newcommand{\Lapm}{\Delta_m}
\newcommand{\calR}{\mathcal R}
\newcommand{\calS}{\mathcal S}
\newcommand{\Am}{\mathcal A_m}
\newcommand{\A}{\mathsf A}
\newcommand{\LL}{\mathcal L}

\title{The $m$-Laplace equation with a gradient term II: classification in the second critical case}
\author{Tian Wu \and Jin Yan \and Hua Zhu}
\date{}

\begin{document}
\maketitle

\renewcommand{\thefootnote}{\fnsymbol{footnote}}

\begin{abstract}
\hspace{2em}Let $1<m<n$ and $0<q<m-1$. We classify positive, locally bounded weak
solutions of
\[
 -\Delta_m u=u^p|Du|^q\quad\text{in }\mathbb R^n,
 \qquad
 p=\frac{m-q}{n-m}\left(n+\frac{q}{m-1-q}\right)-1.
\]
Every solution is constant or belongs to an explicit family of radial
solutions, up to translation and scaling. No global bound, decay, or
finite-energy assumption is required. The proof is based on a scalar
differential identity and a maximum principle that remains valid at
critical points. These lead to a sharp gradient bound, classification
of normalized solutions, and a first-contact argument for arbitrary
entire solutions. Local regularity and the behavior at critical points
are treated throughout in the weak solution class; global $C^2$
regularity is not assumed.

\keywords{$m$-Laplace equation, product gradient term,
Liouville-type theorem}\\
\textbf{2020 Mathematics Subject Classification:}
Primary 35B53; Secondary 35J92, 35B33.
\end{abstract}

\section{Introduction}

In 1971, Obata~\cite{O1971} studied the Yamabe problem
on closed Einstein manifolds and established a rigidity result
using a differential identity. In particular, if a closed
Einstein manifold admits a nonconstant positive conformal
factor that produces a metric of constant scalar curvature,
then the manifold is isometric to the standard sphere,
up to normalization. On the standard sphere, the positive
solutions arise from conformal transformations. Under
stereographic projection, they correspond to the positive
solutions of the critical Yamabe equation
\begin{equation}\label{C-Y:eq}
-\Delta u=u^{\frac{n+2}{n-2}}
\qquad \text{in } \mathbb{R}^n,\quad n\geqslant 3.
\end{equation}
These solutions form the explicit family
\begin{equation}\label{bubble solution}
u_{\lambda,x_0}(x)
=
[n(n-2)]^{\frac{n-2}{4}}
\left(
\frac{\lambda}
{1+\lambda^2|x-x_0|^2}
\right)^{\frac{n-2}{2}},
\qquad
\lambda>0,\quad x_0\in\mathbb{R}^n.
\end{equation}

Gidas, Ni, and Nirenberg~\cite{GNN1981} used the method of
moving planes to establish the radial symmetry of positive
solutions of the critical Yamabe equation under a suitable
decay assumption at infinity. In particular, assuming
\begin{equation}\label{condi}
    u(x)=O(|x|^{2-n})
\qquad \text{as } |x|\to\infty,
\end{equation}
the positive solutions are given by the explicit family \eqref{bubble solution}.
Up to multiplication by a nonzero constant, this family
also characterizes the nonnegative extremals of the sharp
Sobolev inequality. In 1989, Caffarelli, Gidas, and Spruck~\cite{CGS1989}
classified all positive entire classical solutions of the
critical Yamabe equation \eqref{C-Y:eq} without imposing an a priori
decay assumption at infinity. Their approach involved
the Kelvin transform and an asymptotic symmetry method.
Later, Chen and Li~\cite{CL1991} gave another,
more concise proof of this classification result using
the Kelvin transform and the method of moving planes. Chang, Gursky, and Yang~\cite{CGY2003} provided an alternative
proof of the classification of positive entire solutions
of the critical Yamabe equation \eqref{C-Y:eq} by employing a differential
identity inspired by Obata's method and estimating the
boundary terms at infinity. When $n=3$, they showed that
every positive entire smooth solution belongs to the
explicit family~\eqref{bubble solution}, without imposing any additional
decay or integrability assumption. For $n\geqslant 4$, they
obtained the same classification under the additional
finite-volume condition
\[
\int_{\mathbb{R}^n}
u^{\frac{2n}{n-2}}\,dx<\infty.
\]

Correspondingly, we consider the $m$-Laplace equation,
\begin{equation}\label{C-Y:m-eq}
-\Delta_{m} u=u^{p}
\qquad \text{in } \mathbb{R}^n,\quad n\geqslant 3.
\end{equation}
which is a quasilinear counterpart of the Euclidean Yamabe
equation. When $p=m^*-1$, where
\[
m^*=\frac{nm}{n-m}
\]
is the critical Sobolev exponent, equation \eqref{C-Y:m-eq} is called critical $m$-Laplace equation. Under suitable assumptions on the solution class, every positive
solution of the critical $m$-Laplace equation \eqref{C-Y:m-eq} belongs to the explicit family of Aubin--Talenti bubble solutions
\begin{equation*}
u_{\lambda,x_0}(x)
=
C_{n,m}
\left(
\frac{\lambda}
{1+\lambda^{\frac{m}{m-1}}
|x-x_0|^{\frac{m}{m-1}}}
\right)^{\frac{n-m}{m}},
\qquad
\lambda>0,\quad x_0\in\mathbb{R}^n,
\end{equation*}
where
\[
C_{n,m}
=
\left[
n\left(\frac{n-m}{m-1}\right)^{m-1}
\right]^{\frac{n-m}{m^2}}.
\]
These solutions are also nonnegative extremals of the sharp
Sobolev inequality, with an appropriate normalization.
The classification of positive solutions to
\eqref{C-Y:m-eq} was initially studied under
the finite-energy assumption
\[
u\in D^{1,m}(\mathbb{R}^n)
=
\left\{
u\in L^{m^*}(\mathbb{R}^n):
D u\in L^m(\mathbb{R}^n)
\right\}.
\]
The cases $1<m<2$ and $2<m<n$ require different treatments,
due to the singular or degenerate nature of the $m$-Laplace
operator at critical points. Damascelli, Merch\'an, Montoro
and Sciunzi~\cite{Damascelli2014} obtained classification results
in the range $\frac{2n}{n+2}<m<2$ under the finite-energy
assumption, and the range $1<m\leqslant2$ was subsequently
extended by V\'etois \cite{Vetois2024}. For $2<m<n$, Sciunzi~\cite{Sciunzi2016}
established the classification of all positive
$D^{1,m}(\mathbb{R}^n)$-solutions. Together with the previous
results for $1<m\leqslant 2$, this completed the classification
of positive finite-energy solutions for the full range
$1<m<n$. Ciraolo, Figalli and Roncoroni~\cite{CFR2020}
developed a different approach, based on integral estimates
and integral identities rather than the moving planes method,
to the classification of critical $m$-Laplace equations \eqref{C-Y:m-eq}.
Their method not only recovers the classification of positive
finite-energy solutions in the Euclidean setting, but also
extends the classification to critical anisotropic $m$-Laplace
equations induced by smooth norms in general convex cones.

A natural subsequent question is whether the finite-energy
assumption can be removed. Catino, Monticelli and
Roncoroni~\cite{Catin2023} studied positive entire solutions
which may have infinite energy. They obtained an
unconditional rigidity result when $n=2$, and when $n=3$
with $\frac{3}{2}<m<2$, while in the remaining cases they
established the classification under suitable energy-growth
conditions or appropriate control at infinity. More recently, Ou~\cite{Ou2025} obtained the classification
without the finite-energy assumption in a substantially
larger range $n\geqslant 3,\frac{n+1}{3}\leqslant m<n$. V\'etois~\cite{Vetois2024} further
improved the admissible range for $n\geqslant4$. These results
show a progressive development from the classification of
finite-energy solutions toward the classification of
general positive entire solutions of the critical
$m$-Laplace equation. More recently, Zhang~\cite{Zhang2026} established a fully
unrestricted Liouville classification for the critical
$m$-Laplace equation in the full range $1<m<n$, without
imposing any finite-energy, growth, or asymptotic assumption
on positive entire solutions. More generally, the result
applies to critical equations with a positive, bounded,
continuous, and nonincreasing coefficient. As a further
consequence, a scale-invariant Schoen-type Harnack inequality
was established, and, for the pure critical equation, the
Liouville classification was shown to be equivalent to the
corresponding Harnack estimate.

Motivated by the study of critical elliptic equations,
it is natural to consider semilinear equations involving
nonlinearities that depend on both the solution and its
gradient. In particular, the equation
\begin{equation}\label{2-(p,q)-eq}
-\Delta u=u^p|D u|^q
\qquad \text{in } \mathbb{R}^n
\end{equation}
has attracted considerable attention in connection with
Liouville-type theorems, critical exponents, and the
existence and classification of positive entire solutions. 
The first critical curve of equation \eqref{2-(p,q)-eq} is given by
\begin{equation*}
(n-2)p+(n-1)q=n.
\end{equation*}
In the strictly subcritical range $(n-2)p+(n-1)q<n$,
Mitidieri and Pohozaev~\cite{MP2001} established a
Liouville-type result for nonnegative entire $C^{1}$ supersolutions. By means of an integral Bernstein method, Bidaut-V\'eron, Garc\'ia-Huidobro and
V\'eron~\cite{BGHV2019} established
a Liouville-type theorem in a range of parameters
characterized by the condition
\[
G_n(p,q)<0,
\]
where $G_n(p,q)
={}-\bigl((n-1)^2q+n-2\bigr)p^2+\bigl[n(n-1)q^2-(n^2+n-1)q-n-2\bigr]p
-nq^2.$
In particular, every positive entire solution is constant
in this range. They also obtained a sharp existence criterion for
nonconstant positive entire radial solutions: such
solutions exist if and only if $p\geqslant 0$, $0\leqslant q<1$,
and
\begin{equation*}
(n-2)p+(n-1)q
\geqslant n+\frac{2-q}{1-q}.
\end{equation*}
Moreover, they constructed an explicit family of positive
radial solutions on the corresponding second critical curve $(n-2)p+(n-1)q= n+\frac{2-q}{1-q}$.

Compared with the classical Laplace operator, the
$m$-Laplace operator has a more general quasilinear
structure. For $m\neq 2$, the ellipticity of its
linearization depends on the gradient, and the operator
may become degenerate when $m>2$ or singular when
$1<m<2$ at points where the gradient vanishes.
These features make the regularity analysis and
classification of positive solutions more challenging.

A natural quasilinear generalization of this equation \eqref{2-(p,q)-eq} is
\begin{equation}
-\Delta_m u=u^p|D u|^q
\qquad \text{in } \mathbb{R}^n,
\label{equation}
\end{equation}
where
\[
\Delta_m u
=
\operatorname{div}
\left(
|D u|^{m-2}D u
\right),
\qquad 1<m<n.
\]
The interaction between the nonlinear diffusion operator
and the product-type gradient nonlinearity leads to a
more intricate critical structure. Moreover, the
degeneracy or singularity of the $m$-Laplace operator
at critical points introduces additional difficulties
in the analysis of positive solutions. These features motivate the investigation of critical
thresholds, Liouville-type properties, and the existence
and classification of positive entire solutions of
equation~\eqref{equation}.

A positive, locally bounded weak solution of equation~\eqref{equation} in an open set $\Omega$ is a
function
\[
 u\in W^{1,m}_{\mathrm{loc}}(\Omega)
       \cap L^\infty_{\mathrm{loc}}(\Omega),
 \qquad u>0\ \text{almost everywhere},
\]
that satisfies
\begin{equation}\label{weak}
 \int_\Omega \Am(Du)\cdot D\varphi\dd
 =\int_\Omega u^p|Du|^q\varphi\dd,
 \qquad \varphi\in C_c^\infty(\Omega),
\end{equation}
where $\Am(\xi)=|\xi|^{m-2}\xi, \Delta_m u=\Div\Am(Du)$. Since $p>0$ and $q<m-1$, the right-hand side belongs locally to
$L^{m/q}\subset L^{m/(m-1)}$. Thus \eqref{weak} also holds for compactly
supported $W^{1,m}$ test functions, by density.

\begin{theorem}\label{main}
Assume $ n\geqslant2, 1<m<n, 0<q<m-1,p=\frac{m-q}{n-m}\left(n+\frac{q}{m-1-q}\right)-1$. Every positive, locally bounded weak solution of
\eqref{equation} is constant or has the form
\begin{equation}\label{bubblefamily}
 u(x)=\left[
 \frac{C_{n,m,q}\lambda}
 {1+(\lambda|x-x_0|)^{(m-q)/(m-1-q)}}
 \right]^{\frac{(n-m)(m-1-q)}{(m-1)(m-q)}}
\end{equation}
for some $x_0\in\mathbb R^n$ and $\lambda>0$, where
\begin{equation*}\label{Cconstant}
 C_{n,m,q}=
 \left[
 \left(n+\frac{q}{m-1-q}\right)
 \left(\frac{n-m}{m-1}\right)^{m-1-q}
 \right]^{1/(m-q)}.
\end{equation*}
Each solution has a positive $C^{1,\eta}_{\mathrm{loc}}$ representative
for some $\eta>0$ and is smooth wherever its gradient is nonzero.
\end{theorem}

The classification problem for critical $m$-Laplace equations
has also been extensively studied in the half-space setting.
A fundamental model arises from the Euler--Lagrange equation
associated with the Sobolev trace inequality. More precisely,
letting
\[
m_*=\frac{(n-1)m}{n-m},
\]
one is led to the critical boundary problem
\begin{equation}
\begin{cases}
\Delta_m u=0,\quad u>0,
& \text{in } \mathbb{R}_+^n,\\[1mm]
|D u|^{m-2}\dfrac{\partial u}{\partial x_n}
=-u^{m_*-1},
& \text{on } \partial\mathbb{R}_+^n.
\end{cases}
\label{eq:trace-critical}
\end{equation}
More recently, Zhou~\cite{YangZhou2024} classified all positive
finite-energy solutions of \eqref{eq:trace-critical} for the
full range $1<m<n$. A closely related problem arises when the critical Sobolev
nonlinearity is also present in the interior. More generally,
consider the problem
\begin{equation}
\begin{cases}
\Delta_m u+u^p=0,\quad u>0,
& \text{in } \mathbb{R}^n_+,\\[1mm]
|D u|^{m-2}\dfrac{\partial u}{\partial x_n}
=-u^q,
& \text{on } \partial\mathbb{R}^n_+,
\end{cases}
\label{eq:half-space}
\end{equation}
where $1<m<n$. In the critical case,
\[
p=\frac{nm}{n-m}-1,
\qquad
q=\frac{n(m-1)}{n-m},
\]
problem~\eqref{eq:half-space} is related to a Sobolev trace
inequality in the half-space and to the Yamabe problem on
manifolds with boundary.

When $m=2$, Escobar~\cite{Escobar1990} classified positive
solutions in the critical case under an additional decay
assumption at infinity. More precisely, under a hypothesis \eqref{condi}, he showed that every positive solution is of the standard
bubble form
\begin{equation}
u(x',x_n)
=
\left(
\frac{\mu}
{1+\mu^2|x-x_0|^2}
\right)^{\frac{n-2}{2}},
\qquad
\mu>0,\quad x_0\in\mathbb{R}_{-}^{n}.
\end{equation}
In 1995, Li and Zhu~\cite{LZ1995} removed
this additional assumption by means of the moving spheres
method and obtained the classification of all positive
solutions. Later, Li and Zhang~\cite{LiZhang2003}
investigated the corresponding subcritical problem and
established Liouville-type nonexistence results using the
moving planes method.

For general $1<m<n$, Zhou~\cite{YangZhou2024} extended the
critical classification to the quasilinear setting by
developing an approach based on the method of vector fields
and integration by parts. More recently, Yu and
Zhou~\cite{YuZhou2025} considered the corresponding subcritical
problem and established Liouville-type nonexistence results
in a range of subcritical exponents. Their approach modifies
the integral identities used in the critical classification
and shows that, under suitable conditions on $p$ and $q$,
problem~\eqref{eq:half-space} admits no nontrivial
nonnegative $C^2$ solutions in $\mathbb{R}^n_+$.

Critical elliptic equations on Riemannian manifolds have also
been extensively studied in connection with classification and
geometric rigidity. On compact Riemannian manifolds $(M^n,g)$ satisfying
$\operatorname{Ric}_g\geqslant (n-1)g$, consider the equation
\[
\Delta_g u-\lambda u+u^\alpha=0.
\]
Every positive solution must be constant if
\[
1<\alpha\leqslant \frac{n+2}{n-2},
\qquad
0<\lambda\leqslant \frac{n}{\alpha-1},
\]
with the equalities not holding simultaneously, unless
\[
\alpha=\frac{n+2}{n-2},
\qquad
\lambda=\frac{n(n-2)}{4},
\]
and $(M^n,g)$ is isometric to $(\mathbb{S}^n,g_c)$. In this
exceptional case, the equation admits nontrivial solutions of the form
\[
u(x)
=
\left(
\frac{\sqrt{n(n-2)}}
     {2\cosh t+2(\sinh t)\langle a,x\rangle}
\right)^{\frac{n-2}{2}},
\qquad
t\geqslant 0,\quad a\in\mathbb{S}^n.
\]
In the complete noncompact setting, Sun and Wang~\cite{SunWang2025}
recently investigated the critical $m$-Laplace equation
\[
-\Delta_m u=u^{m^*-1}
\qquad \text{on } (M^n,g),
\]
where $1<m<n$, $m^*=nm/(n-m)$, and $(M^n,g)$ has nonnegative
Ricci curvature. They obtained classification results for positive
solutions together with corresponding rigidity results for the
underlying manifold.

\textbf{Scaling and notation:} Four exponents recur in the proof:
\begin{equation}\label{notation}
 \tau=m-1-q,\qquad
 \gamma=1+\frac1\tau,\qquad
 \kappa=\frac{n-m}{m-1},\qquad
 b=\frac{\kappa}{\gamma}.
\end{equation}
Here $\tau$ measures the gap between $q$ and $m-1$; $\gamma$ is the power
in the denominator of \eqref{bubblefamily}; $\kappa$ is the decay
exponent of the $m$-harmonic fundamental solution; and $b$ is the scaling
exponent of \eqref{equation}. In particular,
\begin{equation}\label{naturalscale}
 u_\lambda(x)=\lambda^b u(\lambda x)
\end{equation}
solves the same equation. The critical relation can also be written as
\begin{equation}\label{identities}
 \begin{gathered}
 p=\tau+\frac{m-q}{b}>0,\qquad b\gamma=\kappa,\\
 \kappa p+(\kappa+1)q=n+\gamma,\qquad
 (\kappa+1)(m-1)=n-1.
 \end{gathered}
\end{equation}
Constants denoted by $C$ may change from line to line and depend only on
the indicated data. Balls without a specified center are centered at
the origin.

\textbf{The structure of the proof:} For $m=2$, Dou--Shi--Wu--Zhu \cite{DSWZ} proved a low-dimensional
classification in a restricted range of $q$ by means of an invariant
tensor and its associated scalar function. Our proof follows a scalar
maximum-principle approach. Its use of multipliers and first-contact
rescaling is inspired by Zhang \cite{Zhang2026}; the gradient term in
\eqref{equation} requires different identities and concentration
estimates.

Sections \ref{sec:regularity}--\ref{sec:critical} establish local
regularity, the scalar differential identity, and a maximum principle
across the critical set. A sharp bound for the scalar function then
classifies solutions that attain their global maximum
(Section \ref{sec:normalized}). The first-contact argument in
Section \ref{sec:contact} extends the resulting radial bound to larger
scales. Its two ingredients, a local mass estimate and the expansion
of an isolated $m$-harmonic pole, are stated in
Section \ref{sec:potential}. Finally, Section \ref{sec:harnack}
derives a Harnack estimate, proves that every entire solution is bounded
and attains its maximum, and completes the classification.

We use the classical interior regularity and weak comparison theory
for the $m$-Laplacian, together with the isolated-pole theorem of
Kichenassamy--V\'eron. The versions needed below are stated explicitly.
The arguments involving critical points do not require integration over
the boundary of the critical set.

\section{Local estimates and compactness}
\label{sec:regularity}

\subsection{Interior regularity and compactness}

For later blow-down arguments, we use the rescaled family
\begin{equation}\label{epsilon}
 -\Lapm w=\varepsilon w^p|Dw|^q,
 \qquad \varepsilon\geq0.
\end{equation}
The original equation has $\varepsilon=1$, while the relevant
large-scale rescalings have $\varepsilon\to0$.

For a $C^1$ function $f$, write
\[
 \calR(f)=\{|Df|>0\},
 \qquad
 \calS(f)=\overline{\calR(f)},
\]
where the closure is relative to the domain.

We shall use the standard interior regularity theory for quasilinear equations with $m$-growth; see
Tolksdorf \cite{Tolksdorf} and Lieberman \cite{Lieberman}.
We also use the strong minimum principle for $m$-superharmonic
functions \cite{HKM}.

\begin{proposition}\label{regularity}
Every positive, locally bounded weak solution of \eqref{epsilon}
has a positive $C_{\mathrm{loc}}^{1,\eta}$ representative for some
$\eta>0$, and it is smooth on $\calR(w)$.

Moreover, suppose that $\varepsilon_j$ is bounded and that the
corresponding solutions $w_j$ are locally uniformly bounded.
Then, after passing to a subsequence,
\[
 \varepsilon_j\longrightarrow\varepsilon,
 \qquad
 w_j\longrightarrow w
 \quad\text{in }C_{\mathrm{loc}}^1,
\]
where $w$ is a nonnegative weak solution of the limiting equation.
If the sequence has locally uniform positive lower bounds, then
$w>0$.
\end{proposition}

\begin{proof}
On every compact set where $0\leq w\leq M$, the lower-order term
satisfies
\[
 \varepsilon w^p|\xi|^q
 \leq C_M\bigl(1+|\xi|^{m-1}\bigr),
\]
because $q<m-1$. Standard Caccioppoli estimates provide the required
local energy bounds. The interior $C^{1,\eta}$ estimates cited above
therefore apply, uniformly for locally bounded families.

The continuous representative is nonnegative and $m$-superharmonic.
The strong minimum principle shows that it is either strictly positive
or identically zero on each connected component; the latter is
excluded by positivity almost everywhere. On $\calR(w)$, the equation
is locally uniformly elliptic with smooth coefficients, and standard
bootstrapping yields smoothness. The compactness assertion follows from the Arzelà--Ascoli theorem.
\end{proof}

\subsection{A logarithmic gradient estimate}

\begin{lemma}\label{loggrad}
Let $w>0$ solve \eqref{epsilon} in $B_{2L}$, and let
\[
 M=\sup_{B_{2L}}w.
\]
Then
\begin{equation}\label{logbound}
 \sup_{B_L}|D\log w|
 \leqslant
 C(n,m,q)
 \left(
 L^{-1}+\varepsilon^{1/(m-q)}M^{1/b}
 \right).
\end{equation}
\end{lemma}

\begin{proof}
Set
\[
 f=\log w,
 \qquad
 F=|Df|^2.
\]
On the open set $\{F>0\}$, define the linearized matrix
\[
 \A=I+(m-2)\frac{Df\otimes Df}{F}.
\]
Since $Dw=wDf$, equation \eqref{epsilon} becomes
\begin{equation}\label{logeq}
 \A:D^2f
 =
 -(m-1)F
 -\varepsilon w^{p-\tau}F^{(1-\tau)/2}.
\end{equation}
Here $p-\tau=\frac{m-q}{b}>0$.

Differentiating \eqref{logeq} shows
\begin{equation}\label{logbochner}
 \begin{split}
 &\A:D^2F
 +
 \left[
 2(m-1)
 +(1-\tau)\varepsilon w^{p-\tau}F^{-(m-q)/2}
 \right]Df\cdot DF
 \\
 &\geqslant
 \frac{2}{n+m-2}\bigl(\A:D^2f\bigr)^2
 -C(m)F^{-1}|D^2f\,Df|^2
 -2(p-\tau)\varepsilon w^{p-\tau}F^{(3-\tau)/2}.
 \end{split}
\end{equation}
Indeed, in a frame with $e_1=Df/|Df|$, the quadratic Hessian terms
before applying Cauchy--Schwarz are
\[
 (m-1)f_{11}^2
 +(4-m)\sum_{i>1}f_{1i}^2
 +\sum_{i,j>1}f_{ij}^2.
\]
The diagonal terms control
$\bigl(\A:D^2f\bigr)^2/(n+m-2)$, while the remaining mixed terms are
bounded below by
\[
 -C(m)F^{-1}|D^2f\,Df|^2.
\]

Choose a smooth function $\eta$ supported in $B_{2L}$ such that
\[
 0\leqslant\eta\leqslant1,
 \qquad
 \eta=1\quad\text{in }B_L,\qquad|D\eta|\leqslant C L^{-1},
 \qquad
 |D^2\eta|\leqslant C L^{-2}.
\]
Consider a positive maximum point of $\eta^2F$. At this point,
\begin{equation}\label{maximumrelations}
 DF=-2F\frac{D\eta}{\eta},
 \qquad
 D^2f\,Df=-F\frac{D\eta}{\eta}.
\end{equation}

If $\eta^2F
 \leqslant
 C\varepsilon^{2/(m-q)}M^{2/b}$, the desired estimate already follows. Otherwise, $ \varepsilon w^{p-\tau}F^{-(m-q)/2}$ is sufficiently small at the maximum point. Since
\[
 \varepsilon w^{p-\tau}F^{(3-\tau)/2}
 =
 \left(
 \varepsilon w^{p-\tau}F^{-(m-q)/2}
 \right)F^2
\]
and
\[
 |\A:D^2f|\geqslant(m-1)F
\]
by \eqref{logeq}, the last term in \eqref{logbochner} can be absorbed
by its positive quadratic term.

Using \eqref{maximumrelations}, the cutoff bounds, and
$\A:D^2(\eta^2F)\leqslant0$ at the maximum point, we obtain
\[
 (\eta^2F)^2
 \leqslant
 C\left[
 L^{-2}(\eta^2F)
 +L^{-1}(\eta^2F)^{3/2}
 \right].
\]
Consequently,
\[
 \eta^2F\leqslant C L^{-2}.
\]

Combining the two cases yields
\[
 \eta^2F
 \leqslant
 C\left(
 L^{-2}
 +\varepsilon^{2/(m-q)}M^{2/b}
 \right).
\]
Since $\eta=1$ in $B_L$, taking square roots proves
\eqref{logbound}. The estimate extends to $\{F=0\}$ by continuity.
\end{proof}

Consequently, if $\{w_j\}$ is a family of positive solutions with
locally uniform upper bounds, bounded parameters $\varepsilon_j$, and
$w_j(x_0)\geqslant c_0>0$ at some fixed point $x_0$, then integrating
\eqref{logbound} along a finite chain of overlapping balls yields a
uniform positive lower bound for $w_j$ on every compact subset.

\label{sec:potential}

\subsection{A local mass estimate}

\begin{lemma}\label{massbound}
Let $w>0$ be a $C^1$ weak solution in $B_{8r}(x_0)$ of
\[
 -\Lapm w=\mu\geqslant0,
\]
where $\mu$ is a Radon measure. Then
\begin{equation}\label{morrey}
 \mu(B_r(x_0))
 \leqslant
 C(n,m)r^{n-m}w(x_0)^{m-1}.
\end{equation}
\end{lemma}

\begin{proof}
The weak Harnack inequality for nonnegative $m$-superharmonic
functions \cite{HKM} states that
\begin{equation}\label{weakH}
 \left(
 \frac{1}{|B_{4r}(x_0)|}
 \int_{B_{4r}(x_0)}w^s\dd
 \right)^{1/s}
 \leqslant
 C\inf_{B_{2r}(x_0)}w
 \leqslant Cw(x_0)
\end{equation}
for every
\[
 0<s<\frac{n(m-1)}{n-m}.
\]

Choose
\[
 0<t<\min\left\{m-1,\frac{m}{n-m}\right\}.
\]
Then both $m-1-t$ and $(t+1)(m-1)$ are admissible exponents in
\eqref{weakH}. Let $\eta$ be supported in $B_{3r}(x_0)$, equal to one
on $B_{2r}(x_0)$, and satisfy $|D\eta|\leqslant C/r$. Testing the
equation with $w^{-t}\eta^m$ and applying Young's inequality, we obtain
\[
 \int_{B_{2r}(x_0)}
 w^{-t-1}|Dw|^m\dd
 \leqslant
 Cr^{-m}
 \int_{B_{3r}(x_0)}w^{m-1-t}\dd.
\]
Hölder's inequality and \eqref{weakH} now imply
\[
\begin{aligned}
 \int_{B_{2r}(x_0)}|Dw|^{m-1}\dd
 &\leqslant
 \left(
 \int_{B_{2r}(x_0)}
 w^{-t-1}|Dw|^m\dd
 \right)^{(m-1)/m}
 \\
 &\quad\times
 \left(
 \int_{B_{2r}(x_0)}
 w^{(t+1)(m-1)}\dd
 \right)^{1/m}
 \\
 &\leqslant
 Cr^{n-m+1}w(x_0)^{m-1}.
\end{aligned}
\]
Finally, a cutoff supported in $B_{2r}(x_0)$, equal to one on
$B_r(x_0)$, and with gradient bounded by $C/r$ yields
\[
 \mu(B_r(x_0))
 \leqslant
 \frac{C}{r}
 \int_{B_{2r}(x_0)}|Dw|^{m-1}\dd.
\]
This proves \eqref{morrey}.
\end{proof}

\subsection{The finite part of an isolated
\texorpdfstring{$m$}{m}-harmonic pole}

\begin{lemma}[An isolated $m$-harmonic pole]\label{pole}
Let $W>0$ be $m$-harmonic in $B_L\setminus\{0\}$ and suppose that
\[
 W(x)\leqslant C|x|^{-\kappa}
\]
near the origin. Then either the singularity is removable or
\begin{equation}\label{poleexp}
 W(x)=A|x|^{-\kappa}+\beta+h(x),
 \qquad
 A>0,\quad \beta\in\mathbb R,\quad h(x)\longrightarrow0.
\end{equation}
In the latter case,
\begin{equation}\label{poleder}
 r\sup_{|x|=r}|Dh(x)|\longrightarrow0,
 \qquad
 -\Lapm W
 =
 |\mathbb S^{n-1}|(\kappa A)^{m-1}\delta_0
\end{equation}
in the sense of distributions. The remainder $h$ need not be
$m$-harmonic.
\end{lemma}

\begin{proof}
The expansion \eqref{poleexp} and the first-order asymptotic
\[
 |x|^{\kappa+1}
 \left|
 D\bigl(W-A|x|^{-\kappa}\bigr)
 \right|
 \longrightarrow0
\]
follow from the isolated-singularity theorem of
Kichenassamy--V\'eron \cite[Theorem 1.1 and Remark 1.4]{KV} and \cite{KVerr}; see also
\cite[Proposition 4.3]{Zhang2026}. It remains to sharpen the derivative
estimate for the bounded remainder.

Let
\[
 \Phi(x)=A|x|^{-\kappa}.
\]
Since both $W$ and $\Phi+\beta$ are $m$-harmonic away from the origin,
the function $h=W-\Phi-\beta$ satisfies
\[
 \Div\left[
 \left(
 \int_0^1
 D\Am(D\Phi+sDh)\,ds
 \right)Dh
 \right]=0.
\]
On every fixed annulus $\{\frac12<|y|<2\}$, the rescaled gradients $r^{\kappa+1}DW(ry)$ are uniformly H\"older continuous and converge to $D\bigl(A|y|^{-\kappa}\bigr)$, which is nonzero there. After multiplication by
$r^{(\kappa+1)(m-2)}$, the coefficient matrices in the rescaled
equation for $h(ry)$ are therefore uniformly elliptic with uniform
H\"older bounds.

The interior gradient estimate for linear divergence-form equations
\cite{GT} yields
\[
 \sup_{|y|=1}|D_yh(ry)|
 \leqslant
 C\sup_{1/2<|y|<2}|h(ry)|.
\]
The right-hand side tends to zero by \eqref{poleexp}. Since
\[
 D_yh(ry)=rDh(ry),
\]
the first assertion of \eqref{poleder} follows.

Finally, on $\partial B_r$,
\[
 \Am(DW)\cdot\frac{x}{|x|}
 =
 -(\kappa A)^{m-1}r^{1-n}
 +o(r^{1-n}).
\]
Let $\varphi\in C_c^\infty(B_L)$. Integrating by parts on
$B_R\setminus\overline{B_\varepsilon}$, where
$\operatorname{supp}\varphi\Subset B_R\Subset B_L$, we obtain
\[
 \int_{B_R\setminus B_\varepsilon}
 \Am(DW)\cdot D\varphi\,dx
 =
 -\int_{\partial B_\varepsilon}
 \varphi\,\Am(DW)\cdot\frac{x}{|x|}\,dS.
\]
Substituting the asymptotic formula above into the boundary integral
and letting $\varepsilon\downarrow0$, we conclude that
\[
 -\Lapm W
 =
 |\mathbb S^{n-1}|(\kappa A)^{m-1}\delta_0
\]
in the sense of distributions.
\end{proof}

\section{The scalar identity and maximum principles}
\label{sec:scalar}

\subsection{The transformation and the weighted-trace identity}

Let
\[
v=u^{-1/b},\qquad z=|Dv|,
\]
and define
\begin{equation}\label{coefficients}
B=(m-1)(b+1),\qquad
\mathfrak D=(m-q)B=(m-1)+(n-1)\tau,\qquad
c=b^{-\tau}.
\end{equation}
The critical relation implies
\[
p+1=\frac{\mathfrak D}{(m-1)b}.
\]
We consider
\begin{equation}\label{vP}
P=\frac{Bz^{m-q}+c}{v}.
\end{equation}

The weak equation for $v$ is
\begin{equation}\label{veq}
v\Lapm v=Bz^m+cz^q,
\qquad
\Lapm v=Pz^q.
\end{equation}
On $\{z>0\}$, let
\begin{equation}\label{tensors}
\nu=\frac{Dv}{z},\qquad
\A=I+(m-2)\nu\otimes\nu,\qquad
T=D^2v-\frac{Pz^{1-\tau}}{\mathfrak D}
\bigl(\tau I+(1-\tau)\nu\otimes\nu\bigr).
\end{equation}
In a frame with $e_1=\nu$, equation \eqref{veq} and
differentiation of \eqref{vP} yield
\begin{equation}\label{weightedtrace}
(m-1)T_{11}+\sum_{i>1}T_{ii}=0,
\qquad
vDP=\mathfrak D z^{\tau-1}T Dv.
\end{equation}
Thus $T$ is trace-free with respect to the linearized matrix $\A$.

Define
\begin{equation}\label{operator}
\LL\psi=
z^{2-m}\Div\bigl(z^{m-2}\A D\psi\bigr)+
\left(
\frac{2(m-1)-\mathfrak D}{v}
-qPz^{-(m-q)}
\right)Dv\cdot D\psi.
\end{equation}
This operator is locally uniformly elliptic on $\{z>0\}$.

\begin{proposition}\label{scalaridentity}
On $\{z>0\}$,
\begin{equation}\label{Pidentity}
\LL P=\frac{\mathfrak D}{v}z^{\tau-1}
\bigg[
(m-1)\tau T_{11}^2
+(2(m-1)-q)\sum_{i>1}T_{1i}^2
+\sum_{i,j>1}T_{ij}^2
\bigg].
\end{equation}
In particular,
\begin{equation}\label{Qpositive}
\LL P\geqslant
\frac{\mathfrak D}{v}z^{\tau-1}
\left[
\frac{(m-1)\mathfrak D}{n-1}T_{11}^2
+(2(m-1)-q)\sum_{i>1}T_{1i}^2
\right],
\end{equation}
and hence
\begin{equation}\label{Pgradient}
\LL P\geqslant
vz^{-(m-q)}
\left[
\frac{m-1}{n-1}(\nu\cdot DP)^2
+
\frac{2(m-1)-q}{\mathfrak D}
|DP-(\nu\cdot DP)\nu|^2
\right].
\end{equation}
\end{proposition}

\begin{proof}
For this calculation, denote the divergence part of $\LL$ by
\[
\LL_0\psi
=
z^{2-m}\Div\bigl(z^{m-2}\A D\psi\bigr).
\]
Fix a point in $\{z>0\}$ and choose an orthonormal frame with
$e_1=\nu$. At this point,
\[
v_1=z,\qquad v_i=0\quad(i>1),\qquad z_i=v_{1i}.
\]

Differentiate
\[
\Lapm v=Pz^q
\]
with respect to $x_i$. The linearized equation is
\[
\Div\bigl(z^{m-2}\A Dv_i\bigr)
=
z^qP_i+qPz^{q-1}z_i.
\]
Using this identity in the differentiation of $z^{m-q}$, we obtain
\begin{equation}\label{L0power}
\begin{split}
\LL_0(z^{m-q})
={}&(m-q)\Bigg\{
z^{\tau-1}\bigg[
(m-1)\tau v_{11}^2
+(2(m-1)-q)\sum_{i>1}v_{1i}^2
\\
&\hspace{42mm}
+\sum_{i,j>1}v_{ij}^2
\bigg]
+Dv\cdot DP+qPv_{11}
\Bigg\}.
\end{split}
\end{equation}
Moreover, $\A Dv=(m-1)Dv$, and hence \eqref{veq} implies
\begin{equation}\label{L0v}
\LL_0v=(m-1)Pz^{1-\tau}.
\end{equation}

Apply $\LL_0$ to
\[
vP=Bz^{m-q}+c.
\]
By the product rule,
\[
\LL_0(vP)
=
v\LL_0P+P\LL_0v+2(m-1)Dv\cdot DP.
\]
Since $B(m-q)=\mathfrak D$, equations \eqref{L0power} and
\eqref{L0v} lead to
\begin{equation}\label{Pintermediate}
\begin{split}
v\LL_0P
={}&
\bigl(\mathfrak D-2(m-1)\bigr)Dv\cdot DP
+\mathfrak D qPv_{11}
-(m-1)P^2z^{1-\tau}
\\
&+
\mathfrak D z^{\tau-1}\bigg[
(m-1)\tau v_{11}^2
+(2(m-1)-q)\sum_{i>1}v_{1i}^2
+\sum_{i,j>1}v_{ij}^2
\bigg].
\end{split}
\end{equation}

From the definition of $T$,
\[
v_{11}
=
T_{11}+\frac{Pz^{1-\tau}}{\mathfrak D},
\qquad
v_{1i}=T_{1i}\quad(i>1),\qquad v_{ij}
=
T_{ij}
+\frac{\tau Pz^{1-\tau}}{\mathfrak D}\delta_{ij},
\qquad (i,j>1).
\]
After these expressions are inserted into the square bracket in
\eqref{Pintermediate}, the terms linear in $T$ are
\[
2\tau\frac{Pz^{1-\tau}}{\mathfrak D}
\left(
(m-1)T_{11}+\sum_{i>1}T_{ii}
\right).
\]
They vanish by the identity
\eqref{weightedtrace}.

The terms containing only $P$ combine as
\[
\bigl(\tau+q-(m-1)\bigr)P^2z^{1-\tau},
\]
which is zero because $\tau+q=m-1$. Therefore,
\begin{equation}\label{PbeforeDrift}
\begin{split}
v\LL_0P
={}&
\bigl(\mathfrak D-2(m-1)\bigr)Dv\cdot DP+
\mathfrak D qPT_{11}
\\
&+
\mathfrak D z^{\tau-1}\bigg[
(m-1)\tau T_{11}^2
+(2(m-1)-q)\sum_{i>1}T_{1i}^2
+\sum_{i,j>1}T_{ij}^2
\bigg].
\end{split}
\end{equation}

By the second identity in \eqref{weightedtrace}, we have
\[
\mathfrak D qPT_{11}
=
qPvz^{-(m-q)}Dv\cdot DP.
\]
After division by $v$, the two first-order terms in
\eqref{PbeforeDrift} are precisely absorbed into the drift term in
the definition of $\LL$. We arrive at
\[
\LL P
=
\frac{\mathfrak D}{v}z^{\tau-1}
\bigg[
(m-1)\tau T_{11}^2
+(2(m-1)-q)\sum_{i>1}T_{1i}^2
+\sum_{i,j>1}T_{ij}^2
\bigg],
\]
which proves \eqref{Pidentity}.

Next, \eqref{weightedtrace} implies
\[
\sum_{i>1}T_{ii}=-(m-1)T_{11}.
\]
Consequently,
\[
\sum_{i,j>1}T_{ij}^2
\geqslant
\frac{1}{n-1}
\left(\sum_{i>1}T_{ii}\right)^2
=
\frac{(m-1)^2}{n-1}T_{11}^2.
\]
Together with $\mathfrak D=(m-1)+(n-1)\tau$, this proves \eqref{Qpositive}.

Finally, the identity
\[
vDP=\mathfrak D z^{\tau-1}T Dv
\]
shows that
\[
\nu\cdot DP=\frac{\mathfrak D}{v}z^\tau T_{11},
\qquad
DP-(\nu\cdot DP)\nu
=
\frac{\mathfrak D}{v}z^\tau
\sum_{i>1}T_{1i}e_i.
\]
Substitution into \eqref{Qpositive} proves
\eqref{Pgradient}.
\end{proof}

For a solution of \eqref{epsilon}, define
\begin{equation}\label{pseudoeps}
\ell_\varepsilon[w]
=
w^{-\frac{(m-q)(n-1)}{n-m}}|Dw|^{m-q}
+\frac{b}{B}\varepsilon w^{1/b}.
\end{equation}
For the original equation, write $\ell[u]=\ell_1[u]$. The change of
variables above shows that
\begin{equation}\label{pseudo}
\ell[u]=\frac{b^{m-q}}{B}P.
\end{equation}
This form will be used in the boundary maximum principle and in the
first-contact rescaling.

\subsection{Growth at critical points and a maximum principle across the critical set}
\label{sec:critical}

Proposition \ref{scalaridentity} shows that $P$ is
$\LL$-subharmonic on $\calR(v)$. Since $\LL$ is locally uniformly
elliptic there, the usual strong maximum principle applies away from
the critical set. We now treat maxima attained at points where
$Dv=0$.

\begin{lemma}\label{growth}
Suppose that $g\in C^1(\overline{B_L(y)})$ is smooth on $\calR(g)$ and
\[
 \Lapm g\geqslant\lambda|Dg|^q
 \quad\text{on }\calR(g),
 \qquad
 \lambda>0,
 \qquad
 y\in\calS(g).
\]
Then
\begin{equation}\label{growthbound}
 \sup_{B_L(y)}g
 \geqslant g(y)+K(\lambda)L^\gamma,
 \qquad
 K(\lambda)=
 \left(\frac{\lambda}{B\gamma^{m-q}}\right)^{1/\tau}.
\end{equation}
\end{lemma}

\begin{proof}
The proof follows the standard test-function touching argument from viscosity solution theory. Assume first that $Dg(y)\ne0$, and choose $0<K_0<K(\lambda)$. Define
\[
 \psi(x)=g(y)+K_0|x-y|^\gamma.
\]
Since $\gamma-1=\frac{1}{\tau}$, a direct calculation shows that
\[
 \Lapm\psi
 =
 (K_0\gamma)^{m-1}
 \bigl[n-1+(m-1)(\gamma-1)\bigr]
 |x-y|^{q/\tau}.
\]
Moreover, $n-1+(m-1)(\gamma-1)=B\gamma$. Thus the definition of $K(\lambda)$ and the inequality
$K_0<K(\lambda)$ imply
\[
 \Lapm\psi
 <
 \lambda|D\psi|^q
 \qquad (x\ne y).
\]

Since $\gamma>1$ and $Dg(y)\ne0$, there is a direction along which
$g$ grows linearly away from $y$, whereas $K_0|x-y|^\gamma=o(|x-y|)$. Hence $g(x)-\psi(x)>0$ at some point sufficiently close to $y$. Therefore, the maximum of
$g-\psi$ on $\overline{B_L(y)}$ is positive.

Suppose that this maximum is attained at an interior point $x_0$.
Since $(g-\psi)(y)=0$, we have $x_0\ne y$. At $x_0$,
\[
 Dg(x_0)=D\psi(x_0)
 =K_0\gamma|x_0-y|^{\gamma-2}(x_0-y)\ne0.
\]
Thus $x_0\in\calR(g)$, and both functions are smooth near $x_0$.
Moreover,
\[
 D^2g(x_0)\leqslant D^2\psi(x_0).
\]
Because the two gradients agree and are nonzero, the $m$-Laplace
operator has the same positive definite coefficient matrix for $g$
and $\psi$ at $x_0$. It follows that
\[
 \Lapm g(x_0)\leqslant\Lapm\psi(x_0).
\]
On the other hand,
\[
 \Lapm g(x_0)
 \geqslant\lambda|Dg(x_0)|^q
 =\lambda|D\psi(x_0)|^q
 >\Lapm\psi(x_0),
\]
which is impossible. Thus the positive maximum of $g-\psi$ is attained
on $\partial B_L(y)$. Consequently,
\[
 \sup_{B_L(y)}g
 \geqslant g(y)+K_0L^\gamma.
\]
Letting $K_0\uparrow K(\lambda)$ proves \eqref{growthbound} when
$Dg(y)\ne0$.

For a general point $y\in\calS(g)$, choose
$y_j\in\calR(g)$ such that $y_j\to y$. Applying the result above in $B_{L-|y_j-y|}(y_j)\subset B_L(y)$, we obtain $ \sup_{B_L(y)}g
 \geqslant
 g(y_j)+K(\lambda)\bigl(L-|y_j-y|\bigr)^\gamma$. Passing to the limit as $j\to\infty$ proves
\eqref{growthbound} and completes the proof.
\end{proof}

\begin{corollary}\label{noncollapse}
Let $u_j$ solve \eqref{equation} on balls whose radii tend to infinity,
with
\[
 u_j(0)=1,\qquad 0\in\calS(u_j),
\]
and assume that $u_j$ is locally uniformly bounded from above. Then
every locally convergent subsequence has a nonconstant limit.
\end{corollary}

\begin{proof}
By Lemma \ref{loggrad}, the functions $u_j$ are bounded away from zero
on each fixed ball. Hence
\[
 v_j=u_j^{-1/b}
\]
is uniformly bounded above and below there, and \eqref{veq} implies
\[
 \Lapm v_j\geqslant\lambda|Dv_j|^q
\]
with $\lambda>0$ independent of $j$. Lemma \ref{growth} therefore
provides a uniform positive lower bound for the oscillation of $v_j$
on a fixed ball. Thus no local limit can be constant.
\end{proof}

\begin{lemma}\label{criticallemma}
Let $v>0$ solve the transformed equation \eqref{veq} in a neighborhood
of the origin. Assume that
\[
 v(0)=1,\qquad P\leqslant c,\qquad 0\in\calS(v).
\]
Then $P\equiv c$ in a neighborhood of the origin.
\end{lemma}

\begin{proof}
\medskip
\noindent\textit{Sharp growth and a unique tangent.}
Since
\[
 \frac{c}{v}\leqslant P\leqslant c,
\]
we have $v\geqslant1$. Let
\[
 g=v-1,
 \qquad
 K=K(c)=
 \left(\frac{c}{B\gamma^{m-q}}\right)^{1/\tau}.
\]
The definition of $P$ implies
\begin{equation}\label{sharpgrad}
 B|Dg|^{m-q}\leqslant cg.
\end{equation}
Applying this estimate to $(g+\delta)^{1/\gamma}$ and then letting
$\delta\downarrow0$, we obtain
\begin{equation}\label{sharpgrowth}
 0\leqslant g(x)\leqslant K|x|^\gamma.
\end{equation}

On $B_r$, equation \eqref{veq} implies
\[
 \Lapm g
 \geqslant
 \frac{c}{1+Kr^\gamma}|Dg|^q.
\]
Lemma \ref{growth} and the maximum principle for $m$-subharmonic
functions yield
\begin{equation}\label{tangentmax}
 K\left(\frac{c}{1+Kr^\gamma}\right)r^\gamma
 \leqslant
 \max_{\partial B_r}g
 \leqslant Kr^\gamma.
\end{equation}

For
\[
 W_r(x)=r^{-\gamma}g(rx),
\]
we have
\begin{equation}\label{tangenteq}
 \Lapm W_r
 =
 \frac{Br^\gamma|DW_r|^m+c|DW_r|^q}
 {1+r^\gamma W_r}.
\end{equation}
Equations \eqref{sharpgrad} and \eqref{sharpgrowth} bound $W_r$ and
$DW_r$ on every fixed ball. The right-hand side of
\eqref{tangenteq} is therefore locally bounded, so the standard
interior estimates used in Proposition \ref{regularity} provide
$C^1_{\mathrm{loc}}$ compactness, including at the origin.

Every limit $W$ satisfies
\begin{equation}\label{tangentlimit}
 \Lapm W=c|DW|^q,
 \qquad
 0\leqslant W\leqslant K|x|^\gamma,
 \qquad
 \max_{|x|=1}W=K.
\end{equation}
Let
\[
 \Psi(x)=K|x|^\gamma.
\]
The function $\Psi$ solves the same equation as $W$ and satisfies
$D\Psi\ne0$ away from the origin.

By \eqref{tangentlimit}, $W$ and $\Psi$ touch at some point of the
unit sphere. More generally, let $x_0\ne0$ be any contact point, so that $ W(x_0)=\Psi(x_0)$. Since $W\leqslant\Psi$, the function $\Psi-W$ has a local minimum at
$x_0$. Hence
\[
 DW(x_0)=D\Psi(x_0)\ne0.
\]
By continuity, after shrinking the neighborhood of $x_0$, there is a
constant $\delta>0$ such that
\[
 \bigl|(1-t)D\Psi(x)+tDW(x)\bigr|
 \geqslant\delta
\]
for every $x$ in this neighborhood and every $t\in[0,1]$. It follows that $\Psi-W$ satisfies a uniformly elliptic linear equation
with bounded lower-order coefficients in this neighborhood. Since
\[
 \Psi-W\geqslant0,
 \qquad
 (\Psi-W)(x_0)=0,
\]
the strong minimum principle implies that
\[
 W=\Psi
\]
in a neighborhood of $x_0$. Thus the contact set is open in
$\mathbb R^n\setminus\{0\}$.

The contact set is nonempty by \eqref{tangentlimit} and closed by
continuity. Since $\mathbb R^n\setminus\{0\}$ is connected, it follows
that
\[
 W(x)=\Psi(x)=K|x|^\gamma
 \qquad\text{in }\mathbb R^n\setminus\{0\}.
\]

Thus every tangent limit is the same. Since the rescalings are precompact in $C^1$ on every fixed annulus
and every tangent limit equals $K|x|^\gamma$, the whole family
converges:
\[
 W_r\longrightarrow K|x|^\gamma
 \quad\text{in }C^1
 \left(\left\{\frac12\leqslant|x|\leqslant2\right\}\right)
 \quad\text{as }r\downarrow0.
\]
Rescaling back, we obtain
\begin{equation}\label{tangentC1}
 g(x)=K|x|^\gamma+o(|x|^\gamma),
 \qquad
 Dg(x)=K\gamma|x|^{\gamma-2}x+o(|x|^{\gamma-1}).
\end{equation}
The error terms are uniform in the angular variable. Hence, for all
sufficiently small $x\ne0$,
\[
 g(x)\geqslant\frac K2|x|^\gamma>0,
 \qquad
 |Dg(x)|\geqslant
 \frac{K\gamma}{2}|x|^{\gamma-1}>0.
\]
Thus $g>0$ and $Dg\ne0$ in a sufficiently small punctured ball.

\medskip
\noindent\textit{A punctured-ball barrier.}
Set
\[
 Z=c-P.
\]
Since $Dg\ne0$ in this punctured ball, Proposition
\ref{scalaridentity} yields
\[
 \LL Z=-\LL P\leqslant0.
\]
Moreover, by the definition of $P$, we have
\[
 Z
 =
 c-\frac{Bz^{m-q}+c}{v}
 =
 \frac{cg-B|Dg|^{m-q}}{v}.
\]
By \eqref{sharpgrad}, the numerator is nonnegative. Therefore,
\begin{equation}\label{Zbound}
 \LL Z\leqslant0,
 \qquad
 0\leqslant Z
 =
 \frac{cg-B|Dg|^{m-q}}{v}
 \leqslant\frac{cg}{v}.
\end{equation}
Choose
\begin{equation}\label{sigmarange}
 \max\left\{0,1-\frac{\tau B}{m-1}\right\}<\sigma<1.
\end{equation}
A direct calculation yields
\begin{equation}\label{barrier}
 \LL(g^\sigma)
 =
 \sigma g^{\sigma-2}|Dg|^{1-\tau}
 \bigg[
 \frac{\tau cg}{v}+
 \left(
 \frac{(m-1)(1-b)g}{v}
 +(m-1)(\sigma-1)
 \right)|Dg|^{m-q}
 \bigg].
\end{equation}
For small $g$, the coefficient of $|Dg|^{m-q}$ in the bracket is
negative. Therefore, using the upper bound
\[
 |Dg|^{m-q}\leqslant\frac{cg}{B}
\]
in that term produces a lower bound for the bracket. More precisely,
the bracket is at least
\[
 \frac{cg}{B}
 \left[
 \frac{\tau B+(m-1)(1-b)g}{1+g}
 +(m-1)(\sigma-1)
 \right],
\]
which is positive for sufficiently small $g$ by
\eqref{sigmarange}. Hence
\[
 \LL(g^\sigma)>0
\]
in a sufficiently small punctured ball.

Suppose that $Z$ is not identically zero. Since
\[
 Z\geqslant0,\qquad \LL Z\leqslant0
\]
in the connected punctured ball, the strong minimum principle implies
that $Z>0$ there. Since $g>0$ as well, we may choose
\[
 0<\delta<
 \min_{\partial B_{r_0}}\frac{Z}{g^\sigma},
\]
so that
\[
 Z\geqslant\delta g^\sigma
 \quad\text{on }\partial B_{r_0}.
\]

Both $Z$ and $g^\sigma$ extend continuously to the origin with value
zero. If $Z-\delta g^\sigma$ were negative somewhere, it would attain
a negative minimum at an interior point of
$B_{r_0}\setminus\{0\}$. This is impossible because
\[
 \LL(Z-\delta g^\sigma)
 =
 \LL Z-\delta\LL(g^\sigma)<0,
\]
whereas the value of $\LL(Z-\delta g^\sigma)$ at an interior minimum
must be nonnegative. Hence
\[
 Z\geqslant\delta g^\sigma
 \quad\text{in }B_{r_0}\setminus\{0\}.
\]

On the other hand, \eqref{Zbound} implies
\[
 \frac{Z}{g^\sigma}
 \leqslant
 \frac{c}{v}g^{1-\sigma}
 \longrightarrow0
 \qquad (x\to0),
\]
which contradicts $Z/g^\sigma\geqslant\delta$. Therefore
$Z\equiv0$, and hence $P\equiv c$ near the origin.
\end{proof}

\begin{proposition}\label{Pmaximum}
Let $u$ be a positive solution in a connected open set
$\Omega\subset\mathbb R^n$. If $P$ attains its supremum in $\Omega$,
then it is constant in $\Omega$. In particular, for every bounded
domain $G\Subset\Omega$,
\begin{equation}\label{boundaryMP}
 \max_{\overline G}\ell[u]
 =
 \max_{\partial G}\ell[u].
\end{equation}
For $\varepsilon>0$, the same maximum principles hold for
$\ell_\varepsilon$ along solutions of \eqref{epsilon}.
\end{proposition}

\begin{proof}
Suppose that $M=\sup P$ is attained, and consider the closed set
\[
 \{P=M\}.
\]
At a noncritical point, \eqref{Pidentity} and the usual strong maximum
principle show that this set is locally open. At a point in the
interior of the critical set, $v$ is locally constant, and so is $P$.

At any remaining point $x$ in the maximum set,
\[
 x\in\calS(v),
 \qquad
 v_0=v(x)=\frac{c}{M}.
\]
Define
\[
 \widehat v(y)=v_0^{-1}v(x+v_0y),
 \qquad
 \widehat P(y)=v_0P(x+v_0y).
\]
The transformed equation \eqref{veq} is unchanged by this scaling, and
\[
 \widehat v(0)=1,
 \qquad
 \widehat P\leqslant c,
 \qquad
 \widehat P(0)=c.
\]
Lemma \ref{criticallemma} shows that $\widehat P\equiv c$ near the
origin. Hence the maximum set is open. Since it is also closed, the
connectedness of the domain implies that $P$ is constant.

Because $\ell[u]$ is a positive constant multiple of $P$, this proves
\eqref{boundaryMP}. Finally, if $w$ solves \eqref{epsilon} with
$\varepsilon>0$, then
\[
 \widehat w(y)=w\bigl(\varepsilon^{-1/(m-q)}y\bigr)
\]
solves the coefficient-one equation and
\[
 \ell[\widehat w](y)
 =
 \varepsilon^{-1}
 \ell_\varepsilon[w]
 \bigl(\varepsilon^{-1/(m-q)}y\bigr).
\]
The result for $\ell_\varepsilon$ follows from the coefficient-one
case.
\end{proof}

\subsection{The global sharp bound}

\begin{lemma}\label{globalP}
Let $u$ be an entire solution of \eqref{equation} such that
\[
 0<u\leqslant1.
\]
Then $P\leqslant c$ in $\mathbb R^n$.
\end{lemma}

\begin{proof}
Since $v=u^{-1/b}\geqslant1$, Lemma \ref{loggrad} applied on unit
balls shows that
\[
 |Dv|\leqslant Cv,
 \qquad
 P\leqslant C(1+v^\tau).
\]
Set
\[
 \Theta=\frac{(n-1)\tau}{m-1};
 \qquad
 \Theta-\tau=\kappa\tau>0.
\]
Fix $\Lambda>c$ and $N>0$. For $t\geqslant1$, define
\begin{equation}\label{multiplier}
 \theta(t)=
 \frac{\Theta}{
 1+Nt^{-b}
 \left(\dfrac{\Lambda-c/t}{\Lambda-c}\right)^{\tau(b+1)}
 },
 \qquad
 \phi(t)=
 \exp\left(-\int_1^t\frac{\theta(s)}{s}\,ds\right).
\end{equation}
A direct differentiation yields
\begin{equation}\label{multODE}
 \frac{m-1}{B}
 \left[
 b\theta\left(1-\frac{\theta}{\Theta}\right)-t\theta'
 \right]
 =
 \frac{\tau\theta(1-\theta/\Theta)(c/t)}
 {\Lambda-c/t}>0.
\end{equation}

Suppose that $P\phi(v)$ has an interior maximum larger than
$\Lambda$. At this point $Dv\ne0$, since
\[
 P\phi(v)\leqslant P=\frac{c}{v}\leqslant c
\]
whenever $Dv=0$. The maximum condition implies
\[
 DP=\frac{\theta(v)P}{v}Dv.
\]
Using \eqref{Pgradient} and \eqref{multODE}, we
obtain
\[
\begin{aligned}
\frac{\LL(P\phi(v))}{\phi(v)}
\geqslant&
\frac{P|Dv|^{1-\tau}}{v}
\Bigg[
 \left(
 \frac{(m-1)\theta^2}{n-1}
 -\tau\theta
 \right)P
\\
&+\frac{1}{B}
 \left(
 (n-1)\tau\theta
 -(m-1)\theta^2
 -(m-1)v\theta'
 \right)
 \left(P-\frac{c}{v}\right)
\Bigg]
\\
=&
\frac{P|Dv|^{1-\tau}}{v}
\frac{\tau\theta(1-\theta/\Theta)}{\Lambda-c/v}
\frac{c}{v}(P-\Lambda)
>0,
\end{aligned}
\]
This contradicts the maximum condition.

For fixed $N$,
\[
 \phi(t)\sim C_Nt^{-\Theta}\qquad(t\to\infty),
\]
and hence
\begin{equation}\label{Mdecay}
 P\phi(v)=O(v^{-\kappa\tau})\longrightarrow0
 \qquad(v\to\infty).
\end{equation}
Suppose that
\[
 \sup_{\mathbb R^n}P\phi(v)>\Lambda,
\]
and choose $x_j\in\mathbb R^n$ such that
\[
 P(x_j)\phi(v(x_j))
 \longrightarrow
 \sup_{\mathbb R^n}P\phi(v).
\]
By \eqref{Mdecay}, the sequence $v(x_j)$ is bounded. Hence
$u(x_j)=v(x_j)^{-b}$ is bounded away from zero.

Translate $x_j$ to the origin. Since the translated solutions remain
bounded above by one and have a uniform positive value at the origin,
Lemma \ref{loggrad} and Proposition \ref{regularity} yield, after
passing to a subsequence, a positive entire limit in
$C^1_{\mathrm{loc}}$. The corresponding functions $P\phi(v)$ converge
locally uniformly, and the limit attains the same supremum at the
origin.

This supremum is larger than $\Lambda>c$. The origin of the limiting
solution is therefore noncritical, since $P=c/v\leqslant c$ at every
critical point. The interior maximum argument above now leads to a contradiction. Thus
\[
 P\phi(v)\leqslant\Lambda.
\]
Letting first $N\to\infty$ and then $\Lambda\downarrow c$ proves
$P\leqslant c$.
\end{proof}

\subsection{Rigidity of normalized solutions}
\label{sec:normalized}

\begin{proposition}\label{normalized}
Let $u$ be a nonconstant entire solution of \eqref{equation}. If
\[
 0<u\leqslant1,
 \qquad
 u(0)=1,
\]
then
\begin{equation}\label{U}
 u(x)=U(x):=(1+K|x|^\gamma)^{-b},
 \qquad
 K=\left(\frac{c}{B\gamma^{m-q}}\right)^{1/\tau}.
\end{equation}
\end{proposition}

\begin{proof}
Since $u$ attains its maximum at the origin,
\[
 v(0)=1,\qquad Dv(0)=0,\qquad P(0)=c.
\]
Lemma \ref{globalP} and Proposition \ref{Pmaximum} imply
$P\equiv c$. Proposition \ref{scalaridentity} then yields $T=0$ on
$\calR(v)$.

Let $g=v-1$. On each connected component of $\calR(v)$,
\begin{equation}\label{gequalities}
 |Dg|^{m-q}=\frac{c}{B}g,
 \qquad
 D^2g=
 \frac{c|Dg|^{1-\tau}}{\mathfrak D}
 \bigl(\tau I+(1-\tau)\nu\otimes\nu\bigr).
\end{equation}
Using
\[
 1-\tau+(m-q)\left(\frac{2}{\gamma}-1\right)=0,
\]
a direct calculation yields
\begin{equation}\label{rootidentities}
 D^2\bigl(g^{2/\gamma}\bigr)=2K^{2/\gamma}I,
 \qquad
 \left|D\bigl(g^{2/\gamma}\bigr)\right|^2
 =4K^{2/\gamma}g^{2/\gamma}.
\end{equation}
Thus, on that component,
\[
 g^{2/\gamma}
 =
 K^{2/\gamma}|x-x_0|^2+d
\]
for some $x_0\in\mathbb R^n$ and constant $d$. The second identity in
\eqref{rootidentities} implies $d=0$, so
\[
 g(x)=K|x-x_0|^\gamma.
\]

Let $\Omega$ be the connected component of $\calR(v)$ under
consideration. On $\Omega$,
\[
 g(x)=K|x-x_0|^\gamma,
 \qquad
 Dg(x)=K\gamma|x-x_0|^{\gamma-2}(x-x_0).
\]
Suppose that $ y\in\partial\Omega\setminus\{x_0\}$. Choose $x_j\in\Omega$ such that $x_j\to y$. By the $C^1$ continuity
of $g$,
\[
 Dg(y)
 =
 \lim_{j\to\infty}Dg(x_j)
 =
 K\gamma|y-x_0|^{\gamma-2}(y-x_0)\ne0.
\]
Thus $y\in\calR(v)$. Since $\calR(v)$ is open, a neighborhood of $y$
is contained in $\calR(v)$ and intersects $\Omega$. The maximality of
the connected component $\Omega$ then implies that this neighborhood
is contained in $\Omega$, contradicting $y\in\partial\Omega$.
Therefore,
\[
 \partial\Omega\cap
 \bigl(\mathbb R^n\setminus\{x_0\}\bigr)=\varnothing.
\]
Hence $\Omega$ is both open and closed in
$\mathbb R^n\setminus\{x_0\}$. Since this set is connected,
\[
 \Omega=\mathbb R^n\setminus\{x_0\}.
\]
By continuity,
\[
 g(x)=K|x-x_0|^\gamma
 \qquad\text{in }\mathbb R^n.
\]
Consequently,
\[
 u(x)=(1+K|x-x_0|^\gamma)^{-b}.
\]
This function attains its maximum only at $x_0$. Since $u(0)=1$, we
have $x_0=0$, which proves \eqref{U}.
\end{proof}

\section{Proof of Theorem \ref{main}}

\subsection{The first-contact lemma}
\label{sec:contact}

\medskip
\noindent\textbf{Selection in the presence of constant regions.}

We first record a selection estimate used below. Let $B$ be a ball,
let $f>0$ be $C^1$ near $\overline B$, and suppose that
$z\in\calS(f)\cap B$ maximizes
\[
 \dist(x,\partial B)^\rho f(x)
\]
over $\calS(f)\cap B$, where $\rho>0$. Write
\[
 d_z=\dist(z,\partial B).
\]
Then
\begin{equation}\label{selection}
 f(x)\leqslant
 f(z)\left(1-\frac{|x-z|}{d_z}\right)^{-\rho},
 \qquad |x-z|<d_z.
\end{equation}
Indeed, for $x\in\calS(f)$, maximality and
\[
 \dist(x,\partial B)\geqslant d_z-|x-z|
\]
imply \eqref{selection}. If $x\notin\calS(f)$, then $x$ lies in a
connected component of $B\setminus\calS(f)$, on which $f$ is constant.
The segment from $z$ to $x$ meets the boundary of that component at a
point $y\in\calS(f)$ satisfying
\[
 f(y)=f(x),
 \qquad
 |y-z|\leqslant|x-z|.
\]
Applying the estimate at $y$ proves \eqref{selection}. The weighted
maximum is attained because the weight vanishes on $\partial B$.

For the normalized solution $U$ in \eqref{U}, define
\begin{equation}\label{tailconstants}
 A_*=K^{-b},
 \qquad
 \mathfrak M=
 |\mathbb S^{n-1}|(\kappa A_*)^{m-1}.
\end{equation}
Then
\begin{equation}\label{bubbleidentities}
 U(r)\sim A_*r^{-\kappa},
 \qquad
 \int_{\mathbb R^n}U^p|DU|^q\dd=\mathfrak M,
 \qquad
 \ell[U]\equiv\frac{b}{B}
 =
 \kappa^{m-q}A_*^{-(m-q)/\kappa}.
\end{equation}
The first and third identities follow directly from \eqref{U} and the
definitions. The second follows by integrating the radial equation
over $B_R$ and letting $R\to\infty$.

\begin{proposition}[First-contact control]\label{firstcontact}
Let $\Gamma_j\to\infty$, and let $u_j$ solve \eqref{equation} in
$B_{4\Gamma_j}$ with
\[
 0<u_j\leqslant C_0,
 \qquad
 u_j\longrightarrow U
 \quad\text{in }C^1_{\mathrm{loc}}(\mathbb R^n).
\]
For every fixed $\vartheta>1$ and every sequence
$S_j=o(\Gamma_j)$,
\begin{equation}\label{firstclaim}
 u_j<\vartheta U
 \quad\text{in }B_{S_j}
\end{equation}
for all sufficiently large $j$.
\end{proposition}

\begin{proof}
Suppose that \eqref{firstclaim} fails. By local convergence and
continuity, there are first-contact radii
\[
 \rho_j\to\infty,
 \qquad
 \rho_j=o(\Gamma_j),
\]
and points $e_j\in\mathbb S^{n-1}$ such that
\begin{equation}\label{contact}
 u_j\leqslant\vartheta U
 \quad\text{in }B_{\rho_j},
 \qquad
 u_j(\rho_je_j)=\vartheta U(\rho_j).
\end{equation}
Define
\begin{equation}\label{blowdown}
 W_j(x)=\rho_j^\kappa u_j(\rho_jx),
 \qquad
 \varepsilon_j=\rho_j^{-\gamma},
 \qquad
 \mu_j=\varepsilon_jW_j^p|DW_j|^q\dd.
\end{equation}
The functions $W_j$ are defined on balls whose radii tend to infinity
and satisfy
\begin{equation}\label{blowdowneq}
 -\Lapm W_j=\mu_j,
 \qquad
 W_j\leqslant C_0\varepsilon_j^{-b},
 \qquad
 W_j(e_j)\longrightarrow\vartheta A_*.
\end{equation}
The critical identity
\[
 \kappa p+(\kappa+1)q=n+\gamma
\]
also implies
\[
 \mu_j(E)
 =
 \int_{\rho_jE}u_j^p|Du_j|^q\dd
\]
for every measurable set $E$.

\medskip
\noindent\textit{Core mass and the punctured limit.}
On every compact subset of $B_1\setminus\{0\}$, \eqref{contact}
bounds $W_j$ uniformly. Since $\varepsilon_j\to0$, Proposition
\ref{regularity} yields, after passing to a subsequence,
\[
 W_j\longrightarrow W
 \quad\text{in }C^1_{\mathrm{loc}}(B_1\setminus\{0\}),
\]
where $W\geqslant0$ is $m$-harmonic.

Fix $t<1$. Applying Lemma \ref{loggrad} on balls whose radii are
comparable to $|y|$, and using \eqref{contact}, we obtain
\begin{equation}\label{tailbound}
 \begin{split}
 u_j(y)&\leqslant C_t|y|^{-\kappa},\\
 |Du_j(y)|&\leqslant C_t|y|^{-\kappa-1},\\
 u_j^p|Du_j|^q&\leqslant C_t|y|^{-n-\gamma}
 \end{split}
 \qquad
 (1\leqslant|y|\leqslant t\rho_j).
\end{equation}
Since
\[
 \mu_j(B_t)
 =
 \int_{B_{t\rho_j}}u_j^p|Du_j|^q\dd,
\]
the local convergence $u_j\to U$ and the integrable tail
\eqref{tailbound} imply
\begin{equation}\label{coremass}
 \mu_j(B_t)\longrightarrow\mathfrak M,
 \qquad
 \mu_j\rightharpoonup\mathfrak M\delta_0
 \quad\text{locally in }B_1.
\end{equation}
Radial testing of the weak equation, followed by
$C^1$ convergence on $\partial B_t$, yields
\[
 -\int_{\partial B_t}
 \Am(DW)\cdot\nu_{\partial B_t}\,\ds
 =
 \mathfrak M>0.
\]
The identity first holds for almost every $t$ and then for every
$t\in(0,1)$ by continuity. Hence $W\not\equiv0$, and the strong
minimum principle implies $W>0$.

Lemma \ref{pole}, together with \eqref{tailconstants} and
\eqref{coremass}, now shows that
\begin{equation}\label{Wpole}
 W(x)=\Phi(x)+\beta+h(x),
 \qquad
 \Phi(x)=A_*|x|^{-\kappa},
\end{equation}
where
\[
 \beta\in\mathbb R,
 \qquad
 h(x)=o(1),
 \qquad
 r\sup_{|x|=r}|Dh(x)|=o(1).
\]

\medskip
\noindent\textit{Comparison with the fundamental solution.}
We claim that
\[
 W\geqslant\Phi
 \quad\text{in }B_1\setminus\{0\}.
\]
Fix $\delta>0$. By \eqref{Wpole}, there is a small radius $s>0$
such that
\[
 W\geqslant(1-\delta/2)\Phi
 \quad\text{on }\partial B_s.
\]
For large $j$,
\[
 W_j\geqslant(1-\delta)\Phi
 \quad\text{on }\partial B_s.
\]
Choose $L_j\to\infty$ so that $B_{L_j}$ remains inside the domain of
$W_j$. On $B_{L_j}\setminus\overline{B_s}$, compare $W_j$ with
\[
 (1-\delta)A_*
 \bigl(|x|^{-\kappa}-L_j^{-\kappa}\bigr).
\]
This function is $m$-harmonic, vanishes on $\partial B_{L_j}$, and is
no larger than $W_j$ on $\partial B_s$. The weak comparison principle
therefore applies. Letting $j\to\infty$ and then
$\delta\downarrow0$, and finally allowing $s\downarrow0$, proves the
claim. In particular,
\[
 \beta\geqslant0.
\]

\medskip
\noindent\textit{Eliminating the finite part.}
The rescaling satisfies the exact identity
\begin{equation}\label{ellscale}
 \ell[u_j](\rho_jx)
 =
 \ell_{\varepsilon_j}[W_j](x).
\end{equation}
Indeed,
\[
 \frac{\kappa(n-1)}{n-m}=\kappa+1,
 \qquad
 \frac{\kappa}{b}=\gamma.
\]

Assume that $\beta>0$. From \eqref{Wpole}, uniformly on
$\partial B_r$,
\begin{equation}\label{deficit}
 \begin{split}
 &W^{-\frac{(m-q)(n-1)}{n-m}}|DW|^{m-q}\\
 &\qquad=
 \frac{b}{B}
 \left[
 1-
 \frac{(m-q)(\kappa+1)\beta}{\kappa A_*}r^\kappa
 +o(r^\kappa)
 \right].
 \end{split}
\end{equation}
The derivative estimate in \eqref{Wpole} implies that the relative
error in $|DW|$ is $o(r^\kappa)$, which is the order needed in
\eqref{deficit}.

Fix $r>0$ so small that the right-hand side of \eqref{deficit} is
strictly less than $b/B$ on $\partial B_r$. Proposition
\ref{Pmaximum}, applied to $B_{r\rho_j}$, implies
\[
 \ell
 \leqslant
 \max_{|y|=r\rho_j}\ell[u_j](y)
 =
 \max_{|x|=r}\ell_{\varepsilon_j}[W_j](x).
\]
For this fixed $r$, let $j\to\infty$. By
\eqref{bubbleidentities}, the left-hand side tends to $b/B$, whereas
the limit of the right-hand side is strictly smaller by
\eqref{deficit}. This contradiction proves
\[
 \beta=0.
\]

We next show that $W=\Phi$. The remainder $h$ need not be
$m$-harmonic, so we compare $W$ directly with $\Phi$. If
$W-\Phi$ vanishes at one point, the same local linearization and
strong minimum principle used in Section \ref{sec:critical} show that
the contact set is open. It is also closed, and therefore
$W=\Phi$ throughout the punctured ball.

Otherwise $W>\Phi$. Fix $r_0\in(0,1)$ and let
\[
 d_0=\min_{\partial B_{r_0}}(W-\Phi)>0.
\]
For each $\delta>0$, compare $W$ with
\[
 (1-\delta)\Phi+\frac{d_0}{2}
\]
on $B_{r_0}\setminus\overline{B_s}$, where $s>0$ is sufficiently
small. The comparison holds on $\partial B_{r_0}$ by the definition
of $d_0$, and on $\partial B_s$ because
\[
 W\geqslant\Phi,
 \qquad
 \delta\Phi\longrightarrow\infty
 \quad (s\downarrow0).
\]
Letting $\delta\downarrow0$ implies
\[
 W\geqslant\Phi+\frac{d_0}{2}
\]
near the origin, contradicting $\beta=0$. Hence
\begin{equation}\label{purepole}
 W=\Phi
 \quad\text{in }B_1\setminus\{0\}.
\end{equation}

\medskip
\noindent\textit{Excluding concentration near the contact point.}
The convergence in \eqref{purepole} holds only on compact subsets of
$B_1\setminus\{0\}$ and does not yet control $W_j(e_j)$. Since
$W_j(e_j)$ is bounded, Lemma \ref{massbound} and
\eqref{blowdowneq} imply
\begin{equation}\label{contactmass}
 \mu_j(B_r(e_j))
 \leqslant Cr^{n-m}
\end{equation}
for all sufficiently small fixed $r$, uniformly in $j$.

Define the intrinsic height
\[
 \Theta_j(x)
 =
 \varepsilon_j^{1/(m-q)}W_j(x)^{1/b},
\]
and set
\[
 \mathfrak m_0
 =
 \frac12\int_{B_1}U^p|DU|^q\dd>0.
\]
Choose $r>0$ so small that
\begin{equation}\label{smallcontactmass}
 C(2r)^{n-m}
 <
 C_0^{-(m-1)/\tau}\mathfrak m_0.
\end{equation}

Suppose that
\[
 \sup_{B_r(e_j)}\Theta_j\longrightarrow\infty
\]
along a subsequence. A point with large value may be chosen in
$\calS(W_j)$: if it lies in a constant open component, the segment
joining it to $e_j$ meets the boundary of that component at a point
with the same value.

Maximize
\[
 \dist(x,\partial B_{2r}(e_j))^bW_j(x)
\]
over $\calS(W_j)\cap B_{2r}(e_j)$. Let $z_j$ be a maximum point and
write
\begin{equation}\label{secondscale}
 \begin{gathered}
 H_j=W_j(z_j),
 \qquad
 d_j=\dist(z_j,\partial B_{2r}(e_j)),\\
 \delta_j=
 \varepsilon_j^{-1/(m-q)}H_j^{-1/b},
 \qquad
 V_j(y)=H_j^{-1}W_j(z_j+\delta_jy).
 \end{gathered}
\end{equation}
The maximizing property implies
\[
 \frac{d_j}{\delta_j}
 =
 \varepsilon_j^{1/(m-q)}d_jH_j^{1/b}
 \longrightarrow\infty.
\]
By \eqref{selection},
\[
 V_j(y)
 \leqslant
 \left(
 1-\frac{\delta_j|y|}{d_j}
 \right)^{-b},
 \qquad
 |y|<\frac{d_j}{\delta_j}.
\]

Each $V_j$ solves the coefficient-one equation, satisfies
\[
 V_j(0)=1,
 \qquad
 0\in\calS(V_j),
\]
and converges locally, after passing to a subsequence, to an entire
solution $V$ with
\[
 V(0)=1,
 \qquad
 V\leqslant1.
\]
Corollary \ref{noncollapse} excludes a constant limit, and Proposition
\ref{normalized} implies
\[
 V=U.
\]
Hence, for large $j$,
\[
 \int_{B_1}V_j^p|DV_j|^q\dd\geqslant\mathfrak m_0.
\]

Returning to the $W_j$ scale,
\begin{align}
 \mu_j(B_{\delta_j}(z_j))
 &=
 H_j^{m-1}\delta_j^{n-m}
 \int_{B_1}V_j^p|DV_j|^q\dd
 \notag\\
 &=
 \left(
 \frac{\varepsilon_j^{-b}}{H_j}
 \right)^{(m-1)/\tau}
 \int_{B_1}V_j^p|DV_j|^q\dd
 \notag\\
 &\geqslant
 C_0^{-(m-1)/\tau}\mathfrak m_0.
 \label{massquantum}
\end{align}
Here we used
\[
 H_j\leqslant C_0\varepsilon_j^{-b}
\]
from \eqref{blowdowneq}. Since $\delta_j/d_j\to0$,
\[
 B_{\delta_j}(z_j)\subset B_{2r}(e_j)
\]
for large $j$. This contradicts \eqref{contactmass} and
\eqref{smallcontactmass}. Therefore $\Theta_j$ is uniformly bounded
on a fixed neighborhood of $e_j$.

\medskip
\noindent\textit{Convergence at the contact point.}
Since
\[
 \varepsilon_jW_j^{p+q-(m-1)}
 =
 \Theta_j^{m-q},
\]
Lemma \ref{loggrad} yields a uniform bound for
$|D\log W_j|$ on a smaller neighborhood of $e_j$. Together with
\[
 W_j(e_j)\longrightarrow\vartheta A_*>0,
\]
this provides uniform positive upper and lower bounds there.
Proposition \ref{regularity} therefore permits $C^1$ convergence after
translating the moving points.

Passing to a subsequence, let
\[
 e_j\longrightarrow e\in\mathbb S^{n-1}.
\]
The local limit agrees with $\Phi$ on its overlap with $B_1$ by
\eqref{purepole}. By continuity, its value at $e$ is
\[
 \Phi(e)=A_*.
\]
Consequently,
\[
 W_j(e_j)\longrightarrow A_*,
\]
which contradicts \eqref{blowdowneq}. This completes the proof.
\end{proof}

\subsection{A Harnack estimate}
\label{sec:harnack}

\begin{proposition}\label{harnack}
There exists $C=C(n,m,q)$ such that every positive solution in
$B_{3L}$ satisfies
\begin{equation}\label{harnackeq}
 \left(\sup_{B_L\cap\calR(u)}u\right)
 \left(\inf_{B_{2L}}u\right)^\tau
 \leqslant
 C L^{-\kappa\tau}.
\end{equation}
The first supremum is defined to be zero when
$B_L\cap\calR(u)=\varnothing$.
\end{proposition}

\begin{proof}
The scaling \eqref{naturalscale} and the identity
\[
 b(1+\tau)=\kappa\tau
\]
reduce the proof to $L=1$. Suppose that \eqref{harnackeq} fails. Then
there are positive solutions $u_j$ in $B_3$ such that
\[
 M_jm_j^\tau\longrightarrow\infty,
 \qquad
 M_j=\sup_{B_1\cap\calR(u_j)}u_j,
 \qquad
 m_j=\inf_{B_2}u_j.
\]

Choose $x_j\in\calS(u_j)\cap B_2$ to maximize
\[
 \dist(x,\partial B_2)^{\kappa\tau}u_j(x).
\]
Set
\[
 A_j=u_j(x_j),
 \qquad
 r_j=\frac18\dist(x_j,\partial B_2),
\]
and define
\begin{equation}\label{Hselection}
 \Gamma_j=r_jA_j^{1/b},
 \qquad
 w_j(y)=A_j^{-1}u_j
 \bigl(x_j+A_j^{-1/b}y\bigr).
\end{equation}
The functions $w_j$ solve \eqref{equation} in $B_{4\Gamma_j}$,
satisfy
\[
 w_j(0)=1,
 \qquad
 0\in\calS(w_j),
\]
and are normalized at the selected points.

Let
\[
 \Xi_j=A_jr_j^{\kappa\tau}m_j^\tau.
\]
Since points of $B_1\cap\calR(u_j)$ have distance at least one from
$\partial B_2$, maximality implies
\[
 A_j(8r_j)^{\kappa\tau}\geqslant M_j.
\]
Hence $\Xi_j\to\infty$. Moreover, $A_j\geqslant m_j$ and
\begin{equation}\label{Gamma}
 \Gamma_j^{\kappa\tau}
 =
 r_j^{\kappa\tau}A_j^{m-q}
 =
 \Xi_j\left(\frac{A_j}{m_j}\right)^\tau
 \geqslant\Xi_j.
\end{equation}
In particular, $\Gamma_j\to\infty$.

Applying \eqref{selection} in $B_2$ with exponent $\kappa\tau$, we
obtain
\[
 w_j(y)
 \leqslant
 \left(
 1-\frac{|y|}{8\Gamma_j}
 \right)^{-\kappa\tau},
 \qquad
 |y|<8\Gamma_j.
\]
Thus
\[
 w_j\leqslant2^{\kappa\tau}
 \quad\text{in }B_{4\Gamma_j}.
\]
On every fixed ball the right-hand side above tends to one. Hence each
local limit $w$ satisfies
\[
 w(0)=1,
 \qquad
 0<w\leqslant1.
\]
Proposition \ref{regularity}, Corollary \ref{noncollapse}, and
Proposition \ref{normalized} imply, after passing to a subsequence,
\[
 w_j\longrightarrow U
 \quad\text{in }C^1_{\mathrm{loc}}(\mathbb R^n).
\]

Consider the intermediate radius
\[
 \widehat\Gamma_j
 =
 \Xi_j^{-1/(2\kappa\tau)}\Gamma_j.
\]
Equation \eqref{Gamma} implies
\[
 \widehat\Gamma_j\longrightarrow\infty,
 \qquad
 \widehat\Gamma_j=o(\Gamma_j).
\]
Proposition \ref{firstcontact}, applied with $\vartheta=2$, therefore
shows that
\[
 w_j\leqslant2U
 \quad\text{in }B_{2\widehat\Gamma_j}
\]
for all sufficiently large $j$.

If $|y|=\widehat\Gamma_j$, then
\[
 A_j^{-1/b}|y|
 =
 r_j\Xi_j^{-1/(2\kappa\tau)}
 <r_j,
\]
so the corresponding point
\[
 x_j+A_j^{-1/b}y
\]
lies in $B_2$. Therefore,
\[
 m_j
 \leqslant
 A_jw_j(y)
 \leqslant
 2A_jU(y)
 \leqslant
 CA_j\widehat\Gamma_j^{-\kappa}.
\]
Raising this inequality to the power $\tau$ and multiplying by
$A_jr_j^{\kappa\tau}$, we find
\[
\begin{aligned}
 \Xi_j
 \leqslant
 C A_j^{1+\tau}r_j^{\kappa\tau}
 \widehat\Gamma_j^{-\kappa\tau}=
 C A_j^{m-q}r_j^{\kappa\tau}
 \Gamma_j^{-\kappa\tau}\Xi_j^{1/2}
 =
 C\Xi_j^{1/2},
\end{aligned}
\]
which contradicts $\Xi_j\to\infty$.
\end{proof}

\subsection{Boundedness and attainment of the maximum}
\label{sec:completion}

\begin{proof}[Proof of Theorem \ref{main}]
Assume that $u$ is nonconstant.

\medskip
\noindent\textit{A lower bound at infinity.}
Since $u$ is positive and $m$-superharmonic, comparison on
$B_R\setminus\overline{B_1}$ with
\[
 \left(\min_{\partial B_1}u\right)
 \frac{|x|^{-\kappa}-R^{-\kappa}}{1-R^{-\kappa}}
\]
and passage to the limit $R\to\infty$ show that
\begin{equation}\label{lowerglobal}
 u(x)\geqslant c_1(1+|x|)^{-\kappa}
 \qquad\text{in }\mathbb R^n
\end{equation}
for some $c_1>0$.

\medskip
\noindent\textit{Global boundedness.}
For $L\geqslant1$, \eqref{lowerglobal} implies
\[
 \inf_{B_{2L}}u\geqslant cL^{-\kappa}.
\]
The Harnack estimate \eqref{harnackeq} therefore yields
\[
 \sup_{B_L\cap\calR(u)}u\leqslant C
\]
with $C$ independent of $L$. Hence $u$ is uniformly bounded on
$\calR(u)$ and, by continuity, on $\calS(u)$.

Each connected component of
$\mathbb R^n\setminus\calS(u)$ is an open region on which $u$ is
constant. Since $u$ is nonconstant, such a component is a proper
subset of $\mathbb R^n$ and has a boundary point in $\calS(u)$ with
the same value. Thus the bound extends to the whole space. We have
proved
\[
 M:=\sup_{\mathbb R^n}u<\infty,
 \qquad
 \sup_{\calR(u)}u=M.
\]

\medskip
\noindent\textit{Attainment of the maximum.}
Choose $x_j\in\calR(u)$ such that
\[
 A_j:=u(x_j)\longrightarrow M,
\]
and define
\begin{equation}\label{lastnormalize}
 w_j(y)
 =
 A_j^{-1}u\bigl(x_j+A_j^{-1/b}y\bigr).
\end{equation}
Then $w_j$ solves \eqref{equation} in $\mathbb R^n$ and satisfies
\[
 w_j(0)=1,
 \qquad
 0\in\calS(w_j),
 \qquad
 0<w_j\leqslant\frac{M}{A_j}\longrightarrow1.
\]
Proposition \ref{regularity} and Corollary \ref{noncollapse} yield a
nonconstant entire limit. Proposition \ref{normalized} identifies it
as $U$, so
\[
 w_j\longrightarrow U
 \quad\text{in }C^1_{\mathrm{loc}}(\mathbb R^n)
\]
along a subsequence.

Suppose that $u$ does not attain $M$. Then $|x_j|\to\infty$ after
passing to a subsequence. The fixed point $0$ has normalized
coordinates
\[
 y_j=-A_j^{1/b}x_j,
\]
and hence
\[
 |y_j|\longrightarrow\infty,
 \qquad
 w_j(y_j)
 =
 \frac{u(0)}{A_j}
 \longrightarrow
 \frac{u(0)}{M}>0.
\]
Since the functions $w_j$ are entire, Proposition \ref{firstcontact}
may be applied with
\[
 \Gamma_j=j(1+|y_j|),
 \qquad
 S_j=2|y_j|.
\]
Here $\Gamma_j\to\infty$ and $S_j=o(\Gamma_j)$. With
$\vartheta=2$, the proposition implies
\[
 w_j(y_j)\leqslant2U(y_j)\longrightarrow0,
\]
a contradiction. Therefore, $u$ attains its maximum at some
$x_0\in\mathbb R^n$.

\medskip
\noindent\textit{Restoring the scale.}
The function
\[
 M^{-1}u\bigl(x_0+M^{-1/b}y\bigr)
\]
satisfies the assumptions of Proposition \ref{normalized}. Hence
\[
 u(x)
 =
 M\left(
 1+KM^{\gamma/b}|x-x_0|^\gamma
 \right)^{-b}.
\]
The definitions of $K$ and $C_{n,m,q}$ imply
\[
 K=C_{n,m,q}^{-\gamma}.
\]
Taking
\[
 \lambda=\frac{M^{1/b}}{C_{n,m,q}}
\]
yields \eqref{bubblefamily}.

Positive constant functions also solve \eqref{equation}, since
$q>0$. The regularity statement follows from Proposition
\ref{regularity}. This completes the proof.
\end{proof}

\noindent\textbf{Acknowledgments:}
Tian Wu was supported by the National Natural Science Foundation
of China (Grant No. 12601391) and the Fundamental Research Funds for
the Central Universities (Grant No. WK0010250106).
Hua Zhu was supported by the National Natural Science Foundation of
China (Grant No. 12501273, 12661041) and the Research Foundation
of Southwest University of Science and Technology
(Grant No. 25zx7153). Tian Wu and Hua Zhu were also supported by
the Open Research Fund of Hubei Key Laboratory of Mathematical
Sciences (Grant No. MPL2026ORG005).

%\noindent\textbf{Research ethics:} Not applicable.

%\noindent\textbf{Informed consent:} Not applicable.

%\noindent\textbf{Author contributions:}

\noindent\textbf{Use of Large Language Models, AI and Machine Learning Tools:}
We used ChatGPT (OpenAI) as an auxiliary tool in preparing this manuscript. ChatGPT assisted us in carrying out computations and developing arguments within the framework of our ideas and proofs. The authors have carefully checked all AI-generated content and take full responsibility for the mathematical statements, proofs, references, and conclusions presented in this manuscript. We have also made every effort to improve the clarity and readability of the manuscript.

%\noindent\textbf{Conflict of interest:}
The authors state no conflict of interest.

%\noindent\textbf{Data availability:} Not applicable.

%\printbibliography[heading=bibintoc, title=\ebibname]
%\appendix
\bibliographystyle{amsplain}
\bibliography{references}

@article{BGHV2019,
  author  = {Bidaut-V{\'e}ron, M.-F. and Garc{\'i}a-Huidobro, M. and V{\'e}ron, L.},
  title   = {{Estimates of solutions of elliptic equations with a source reaction term involving the product of the function and its gradient}},
  journal = {Duke Math. J.},
  volume  = {168},
  number  = {8},
  year    = {2019},
  pages   = {1487--1537},
  doi     = {10.1215/00127094-2018-0067},
  url     = {https://doi.org/10.1215/00127094-2018-0067},
  note    = {\href{https://doi.org/10.1215/00127094-2018-0067}{doi:10.1215/00127094-2018-0067}}
}

@article{CGS1989,
  author  = {Caffarelli, L. A. and Gidas, B. and Spruck, J.},
  title   = {{Asymptotic symmetry and local behavior of semilinear elliptic equations with critical Sobolev growth}},
  journal = {Comm. Pure Appl. Math.},
  volume  = {42},
  number  = {3},
  year    = {1989},
  pages   = {271--297},
  doi     = {10.1002/cpa.3160420304},
  url     = {https://doi.org/10.1002/cpa.3160420304},
  note    = {\href{https://doi.org/10.1002/cpa.3160420304}{doi:10.1002/cpa.3160420304}}
}

@article{Catin2023,
  author    = {Catino, G. and Monticelli, D. D. and Roncoroni, A.},
  title     = {{On the critical $p$-Laplace equation}},
  journal   = {Adv. Math.},
  volume    = {433},
  year      = {2023},
  pages     = {Paper No.~109331},
  eid       = {109331},
  pagetotal = {38},
  doi       = {10.1016/j.aim.2023.109331},
  url       = {https://doi.org/10.1016/j.aim.2023.109331},
  note      = {38 pp. \href{https://doi.org/10.1016/j.aim.2023.109331}{doi:10.1016/j.aim.2023.109331}}
}

@incollection{CGY2003,
  author    = {Chang, S.-Y. A. and Gursky, M. J. and Yang, P. C.},
  title     = {{Entire solutions of a fully nonlinear equation}},
  booktitle = {Lectures on Partial Differential Equations},
  series    = {New Stud. Adv. Math.},
  volume    = {2},
  publisher = {International Press},
  address   = {Somerville, MA},
  year      = {2003},
  pages     = {43--60}
}

@article{CL1991,
  author  = {Chen, W. and Li, C.},
  title   = {{Classification of solutions of some nonlinear elliptic equations}},
  journal = {Duke Math. J.},
  volume  = {63},
  number  = {3},
  year    = {1991},
  pages   = {615--622},
  doi     = {10.1215/S0012-7094-91-06325-8},
  url     = {https://doi.org/10.1215/S0012-7094-91-06325-8},
  note    = {\href{https://doi.org/10.1215/S0012-7094-91-06325-8}{doi:10.1215/S0012-7094-91-06325-8}}
}

@article{CFR2020,
  author  = {Ciraolo, G. and Figalli, A. and Roncoroni, A.},
  title   = {{Symmetry results for critical anisotropic $p$-Laplacian equations in convex cones}},
  journal = {Geom. Funct. Anal.},
  volume  = {30},
  number  = {3},
  year    = {2020},
  pages   = {770--803},
  doi     = {10.1007/s00039-020-00535-3},
  url     = {https://doi.org/10.1007/s00039-020-00535-3},
  note    = {\href{https://doi.org/10.1007/s00039-020-00535-3}{doi:10.1007/s00039-020-00535-3}}
}

@article{Damascelli2014,
  author  = {Damascelli, L. and Merch{\'a}n, S. and Montoro, L. and Sciunzi, B.},
  title   = {{Radial symmetry and applications for a problem involving the $-\Delta_p(\cdot)$ operator and critical nonlinearity in $\mathbb R^N$}},
  journal = {Adv. Math.},
  volume  = {265},
  year    = {2014},
  pages   = {313--335},
  doi     = {10.1016/j.aim.2014.08.004},
  url     = {https://doi.org/10.1016/j.aim.2014.08.004},
  note    = {\href{https://doi.org/10.1016/j.aim.2014.08.004}{doi:10.1016/j.aim.2014.08.004}}
}

@misc{DSWZ,
  author        = {Dou, J. and Shi, B. and Wu, T. and Zhu, H.},
  title         = {{Classification of positive solutions to a class of Laplace equations with a gradient term}},
  howpublished  = {arXiv:2511.20205v1},
  year          = {2025},
  month         = {25 November},
  date          = {2025-11-25},
  eprint        = {2511.20205v1},
  archivePrefix = {arXiv},
  primaryClass  = {math.AP},
  doi           = {10.48550/arXiv.2511.20205},
  url           = {https://arxiv.org/abs/2511.20205v1},
  note          = {\href{https://doi.org/10.48550/arXiv.2511.20205}{doi:10.48550/arXiv.2511.20205}}
}

@article{Escobar1990,
  author  = {Escobar, J. F.},
  title   = {{Uniqueness theorems on conformal deformation of metrics, Sobolev inequalities, and an eigenvalue estimate}},
  journal = {Comm. Pure Appl. Math.},
  volume  = {43},
  number  = {7},
  year    = {1990},
  pages   = {857--883},
  doi     = {10.1002/cpa.3160430703},
  url     = {https://doi.org/10.1002/cpa.3160430703},
  note    = {\href{https://doi.org/10.1002/cpa.3160430703}{doi:10.1002/cpa.3160430703}}
}

@incollection{GNN1981,
  author    = {Gidas, B. and Ni, W.-M. and Nirenberg, L.},
  title     = {{Symmetry of positive solutions of nonlinear elliptic equations in $\mathbb R^n$}},
  booktitle = {Mathematical Analysis and Applications, Part A},
  series    = {Adv. Math. Suppl. Stud.},
  volume    = {7A},
  publisher = {Academic Press},
  address   = {New York},
  year      = {1981},
  pages     = {369--402}
}

@book{GT,
  author    = {Gilbarg, D. and Trudinger, N. S.},
  title     = {{Elliptic Partial Differential Equations of Second Order}},
  series    = {Classics in Mathematics},
  publisher = {Springer},
  address   = {Berlin},
  year      = {2001}
}

@book{HKM,
  author    = {Heinonen, J. and Kilpel{\"a}inen, T. and Martio, O.},
  title     = {{Nonlinear Potential Theory of Degenerate Elliptic Equations}},
  publisher = {Dover Publications},
  address   = {Mineola, NY},
  year      = {2006},
  note      = {Unabridged republication of the 1993 original}
}

@article{KV,
  author  = {Kichenassamy, S. and V{\'e}ron, L.},
  title   = {{Singular solutions of the $p$-Laplace equation}},
  journal = {Math. Ann.},
  volume  = {275},
  number  = {4},
  year    = {1986},
  pages   = {599--615},
  doi     = {10.1007/BF01459140},
  url     = {https://doi.org/10.1007/BF01459140},
  note    = {\href{https://doi.org/10.1007/BF01459140}{doi:10.1007/BF01459140}}
}

@article{KVerr,
  author  = {Kichenassamy, S. and V{\'e}ron, L.},
  title   = {{Erratum: Singular solutions of the $p$-Laplace equation}},
  journal = {Math. Ann.},
  volume  = {277},
  number  = {2},
  year    = {1987},
  pages   = {352}
}

@article{LiZhang2003,
  author  = {Li, Y. Y. and Zhang, L.},
  title   = {{Liouville-type theorems and Harnack-type inequalities for semilinear elliptic equations}},
  journal = {J. Anal. Math.},
  volume  = {90},
  year    = {2003},
  pages   = {27--87},
  doi     = {10.1007/BF02786551},
  url     = {https://doi.org/10.1007/BF02786551},
  note    = {\href{https://doi.org/10.1007/BF02786551}{doi:10.1007/BF02786551}}
}

@article{LZ1995,
  author  = {Li, Y. Y. and Zhu, M.},
  title   = {{Uniqueness theorems through the method of moving spheres}},
  journal = {Duke Math. J.},
  volume  = {80},
  number  = {2},
  year    = {1995},
  pages   = {383--417},
  doi     = {10.1215/S0012-7094-95-08016-8},
  url     = {https://doi.org/10.1215/S0012-7094-95-08016-8},
  note    = {\href{https://doi.org/10.1215/S0012-7094-95-08016-8}{doi:10.1215/S0012-7094-95-08016-8}}
}

@article{Lieberman,
  author  = {Lieberman, G. M.},
  title   = {{Boundary regularity for solutions of degenerate elliptic equations}},
  journal = {Nonlinear Anal.},
  volume  = {12},
  number  = {11},
  year    = {1988},
  pages   = {1203--1219},
  doi     = {10.1016/0362-546X(88)90053-3},
  url     = {https://doi.org/10.1016/0362-546X(88)90053-3},
  note    = {\href{https://doi.org/10.1016/0362-546X(88)90053-3}{doi:10.1016/0362-546X(88)90053-3}}
}

@article{MP2001,
  author  = {Mitidieri, E. and Pokhozhaev, S. I.},
  title   = {{A priori estimates and the absence of solutions of nonlinear partial differential equations and inequalities}},
  journal = {Proc. Steklov Inst. Math.},
  volume  = {234},
  year    = {2001},
  pages   = {1--362}
}

@article{O1971,
  author  = {Obata, M.},
  title   = {{The conjectures on conformal transformations of Riemannian manifolds}},
  journal = {J. Differential Geom.},
  volume  = {6},
  number  = {2},
  year    = {1971/72},
  pages   = {247--258},
  doi     = {10.4310/jdg/1214430407},
  url     = {https://doi.org/10.4310/jdg/1214430407},
  note    = {\href{https://doi.org/10.4310/jdg/1214430407}{doi:10.4310/jdg/1214430407}}
}

@article{Ou2025,
  author  = {Ou, Q.},
  title   = {{On the classification of entire solutions to the critical $p$-Laplace equation}},
  journal = {Math. Ann.},
  volume  = {392},
  number  = {2},
  year    = {2025},
  pages   = {1711--1729},
  doi     = {10.1007/s00208-025-03141-6},
  url     = {https://doi.org/10.1007/s00208-025-03141-6},
  note    = {\href{https://doi.org/10.1007/s00208-025-03141-6}{doi:10.1007/s00208-025-03141-6}}
}

@article{Sciunzi2016,
  author  = {Sciunzi, B.},
  title   = {{Classification of positive $D^{1,p}(\mathbb R^N)$-solutions to the critical $p$-Laplace equation in $\mathbb R^N$}},
  journal = {Adv. Math.},
  volume  = {291},
  year    = {2016},
  pages   = {12--23},
  doi     = {10.1016/j.aim.2015.12.028},
  url     = {https://doi.org/10.1016/j.aim.2015.12.028},
  note    = {\href{https://doi.org/10.1016/j.aim.2015.12.028}{doi:10.1016/j.aim.2015.12.028}}
}

@misc{SunWang2025,
  author        = {Sun, L. and Wang, Y.},
  title         = {{Critical quasilinear equations on Riemannian manifolds}},
  howpublished  = {arXiv:2502.08495v2},
  year          = {2025},
  eprint        = {2502.08495v2},
  archivePrefix = {arXiv},
  primaryClass  = {math.DG},
  doi           = {10.48550/arXiv.2502.08495},
  url           = {https://arxiv.org/abs/2502.08495v2},
  note          = {\href{https://doi.org/10.48550/arXiv.2502.08495}{doi:10.48550/arXiv.2502.08495}}
}

@article{Tolksdorf,
  author  = {Tolksdorf, P.},
  title   = {{Regularity for a more general class of quasilinear elliptic equations}},
  journal = {J. Differential Equations},
  volume  = {51},
  number  = {1},
  year    = {1984},
  pages   = {126--150},
  doi     = {10.1016/0022-0396(84)90105-0},
  url     = {https://doi.org/10.1016/0022-0396(84)90105-0},
  note    = {\href{https://doi.org/10.1016/0022-0396(84)90105-0}{doi:10.1016/0022-0396(84)90105-0}}
}

@article{Vetois2024,
  author  = {V{\'e}tois, J.},
  title   = {{A note on the classification of positive solutions to the critical $p$-Laplace equation in $\mathbb R^n$}},
  journal = {Adv. Nonlinear Stud.},
  volume  = {24},
  number  = {3},
  year    = {2024},
  pages   = {543--552},
  doi     = {10.1515/ans-2023-0129},
  url     = {https://doi.org/10.1515/ans-2023-0129},
  note    = {\href{https://doi.org/10.1515/ans-2023-0129}{doi:10.1515/ans-2023-0129}}
}

@misc{YuZhou2025,
  author        = {Yu, B. and Zhou, Y.},
  title         = {{Liouville type theorem for a class of quasilinear $p$-Laplace type equations in the half space}},
  howpublished  = {arXiv:2509.11283v1},
  year          = {2025},
  eprint        = {2509.11283v1},
  archivePrefix = {arXiv},
  primaryClass  = {math.AP},
  doi           = {10.48550/arXiv.2509.11283},
  url           = {https://arxiv.org/abs/2509.11283v1},
  note          = {\href{https://doi.org/10.48550/arXiv.2509.11283}{doi:10.48550/arXiv.2509.11283}}
}

@misc{Zhang2026,
  author        = {Zhang, Y. R.-Y.},
  title         = {{Critical $p$-Laplace equations with monotone coefficients: Liouville classification and a Schoen-type Harnack inequality}},
  howpublished  = {arXiv:2608.09113v2},
  year          = {2026},
  eprint        = {2608.09113v2},
  archivePrefix = {arXiv},
  primaryClass  = {math.AP},
  doi           = {10.48550/arXiv.2608.09113},
  url           = {https://arxiv.org/abs/2608.09113v2},
  note          = {66 pp. \href{https://doi.org/10.48550/arXiv.2608.09113}{doi:10.48550/arXiv.2608.09113}}
}

@misc{YangZhou2024,
  author        = {Zhou, Y.},
  title         = {{Classification theorem for positive critical points of Sobolev trace inequality}},
  howpublished  = {arXiv:2402.17602v4},
  year          = {2024},
  eprint        = {2402.17602v4},
  archivePrefix = {arXiv},
  primaryClass  = {math.AP},
  doi           = {10.48550/arXiv.2402.17602},
  url           = {https://arxiv.org/abs/2402.17602v4},
  note          = {\href{https://doi.org/10.48550/arXiv.2402.17602}{doi:10.48550/arXiv.2402.17602}}
}

\footnotesize{
    Contact information:
    \begin{itemize}
       
        \item Tian Wu, School of Mathematical Sciences, University of Science and Technology of China, Hefei, Anhui, 230026, People's Republic of China. Email: \emph{wt1997@ustc.edu.cn}

        \item Jin Yan, Institute of Mathematics, Academy of Mathematics and Systems Science, Chinese Academy of Sciences, Beijing, 100190, People's Republic of China. Email: \emph{yanjin@amss.ac.cn}
        
        \item Hua Zhu, School of Mathematical and Physics, Southwest University of Science and Technology, Mianyang, Sichuan, 621010, People's Republic of China. Email: \emph{zhuhmaths@mail.ustc.edu.cn}
    \end{itemize}
}

\end{document}